\documentclass[11pt,a4paper]{amsart}
\usepackage[english]{babel}
\usepackage{enumerate}
\usepackage{amsmath,amssymb,mathtools}
\usepackage[dvips]{graphicx}
\usepackage[latin1]{inputenc}
\usepackage{lmodern}
\usepackage[T1]{fontenc}
\usepackage[latin1]{inputenc}

\usepackage{amsmath, amssymb, amsfonts, graphics, graphicx, float}
\usepackage{multirow}
\usepackage{bigdelim}
\usepackage{parskip}
\usepackage{enumerate}
\usepackage{csquotes}
\usepackage{pgf,tikz}
\usepackage{listings}
\newcommand{\abb}[5]{%
\setlength{\arraycolsep}{0.4ex}%
\begin{array}{rcccc}%
#1 &:\,& #2 & \,\,\longrightarrow\,\, & #3 \\[0.5ex]%
     & & #4 & \longmapsto & #5%
\end{array}%
}

\usepackage{wasysym}
\usepackage{amssymb,amsfonts, color,epsf}
\usepackage{verbatim}
\usepackage{fullpage}
\usepackage{ytableau}
\usepackage{young}
\usepackage{youngtab}
\ifx\pdfpageheight\undefined
   \usepackage[dvips,colorlinks=true,linkcolor=blue,citecolor=green,%
      urlcolor=green]{hyperref}
   \usepackage[dvips]{graphicx}
   \makeatletter
   \edef\Gin@extensions{\Gin@extensions,.mps}
   \makeatother
\else

  \usepackage[bookmarksopen=false,pdftex=true,breaklinks=true,%
      backref=page,pagebackref=true,plainpages=false,%
      hyperindex=true,pdfstartview=FitH,colorlinks=true,%
      pdfpagelabels=true,colorlinks=true,linkcolor=blue,%
      citecolor=red,urlcolor=green,hypertexnames=false%
      ]%
   {hyperref}
\fi
\usepackage{bm}

\usepackage{amsfonts,mathrsfs}
\usepackage{bbm,amsmath,amssymb,amsfonts,graphicx,color,subfig,mathtools, verbatim}

\usepackage{amsmath}
\usepackage{tabularx}
\usepackage{color, colortbl}
\usepackage{url}
\makeindex

\numberwithin{equation}{section}
\newtheorem{thm}{Theorem}
\newtheorem{prop}[thm]{Proposition}
\newtheorem{lemma}[thm]{Lemma}
\newtheorem{cor}[thm]{Corollary}

\newtheorem{definition}[thm]{Definition}
\newtheorem{remark}[thm]{Remark}
\theoremstyle{definition}

\newtheorem{example}[thm]{Example}

\definecolor{RED}{rgb}{0.6,0,0}
\newcommand{\N}{\mathbb{N}}

\newcommand{\R}{\mathbb{R}}

\newcommand{\K}{\mathbb{K}}

\DeclareMathOperator{\Tr}{Tr}

\DeclareMathOperator{\Hom}{Hom}

\DeclareMathOperator{\Gl}{Gl}

\DeclareMathOperator{\ch}{ch}
\DeclareMathOperator{\SYT}{SYT}
\DeclareMathOperator{\YT}{YT}
\DeclareMathOperator{\ev}{ev}
\DeclareMathOperator{\cone}{cone}

\numberwithin{thm}{section}
\newtheoremstyle{break}  
  {3pt}   
  {5pt}   
  {\normalfont}  
  {0pt}       
  {\scshape} 
  {}         
  {4pt}  
  {}          
\theoremstyle{break}

\numberwithin{subcase}{case}

\begin{document}
\title{Reflection groups and cones of sums of squares}

\author{Sebastian Debus}
\address{Department of Mathematics and Statistics, 
UiT -- The Arctic University of Norway, 9037 Troms\o}

\author{Cordian  Riener}
\address{Department of Mathematics and Statistics, 
UiT -- The Arctic University of Norway, 9037 Troms\o}

\thanks{This work has been supported by European Union's Horizon 2020 research and innovation program under the Marie Sk\l{}odowska-Curie grant agreement 813211 (POEMA) and the Troms\o~ Reserach Foundation grant agreement 17matteCR
 }

\begin{abstract}
We consider cones of real forms which are sums of squares and  invariant under a (finite) reflection group. We show how the representation theory of these groups allows to use the symmetry  inherent in these cones to give more efficient descriptions.  We focus especially on the $A_{n}$, $B_n$, and $D_n$ case where we use so-called higher Specht polynomials \cite{ariki1997higher} to give a uniform description of these cones. These descriptions allow us, for example, to study the connection of these cones to non-negative forms. In particular, we give a new proof of a result by Harris \cite{harris1999real} who showed that every non-negative ternary  even symmetric octic form is a sum of squares. 
\end{abstract}
\maketitle
\section{Introduction}
A real form  (homogeneous polynomial) $f\in\R[X_1,\ldots,X_n]$ is called 
a \emph{sum of squares} if it admits  a representation in the form $f=f_1^2+\ldots+ f_m^2$ for some real forms $f_1,\ldots,f_m\in\R[X_1,\ldots, X_n]$ and it is called \emph{positive semidefinite} or \emph{non-negative} if it assumes only non-negative values on $\R^n$.
We will denote  by $\Sigma_{n,2d}$ the cone of sums of squares forms in $n$ variables of degree $2d$ and by $\mathcal{P}_{n,2d}$ the corresponding  cone of non-negative forms. Clearly, every sum of squares is also non-negative, and we therefore have the inclusion $\Sigma_{n,2d}\subset\mathcal{P}_{n,2d}$. Hilbert \cite{hilbert1888darstellung} addressed and solved the question to characterize the cases, when the two cones coincide. As it turns out this only seldom happens, namely only in the case of bivariate forms $(n=2)$, quadratic forms $(2d=2)$, and ternary quartics $(n=3,2d=4)$. Sums of squares play a fundamental role in real algebraic geometry and have in the last two decades become also a very important tool for polynomial optimisation (see for example \cite{scheiderer2009positivity}). Several authors have considered  situations in which one supposes that the forms are invariant under the action of a group: For a group $G\subset \Gl_n(\R)$ we denote by $\mathcal{P}_{n,2d}^G$ and $\Sigma_{n,2d}^G$ the invariant forms in the respective cones. Since this additional requirement can shrink the dimensions of the cones, their study may become more tractable. Furthermore, as presented in \cite{gatermann2004symmetry}, representation theory of groups can be  particularly used to simplify the sums of squares decomposition. Building on this, it was found in \cite{riener2013exploiting,phdthesis} that sums of squares invariant under the symmetric group are highly structured, and the complexity of a sum of squares decomposition in this case stabilizes with $n>2d$.  Furthermore, symmetric sums of squares appear quite naturally in various contexts (for example \cite{raymond2018symmetric}). This makes these cones an interesting object of study. Choi and Lam \cite{choi1977old} initiated a systematic study of Hilbert's classification restricted to the case of symmetric forms, and in a collaboration  with Reznick they further provided a complete study of the cone of  even symmetric sextics \cite{choi1987even}. Whereas they could show that in the sextic case there exists a form which is non-negative but not a sum of squares Harris \cite{harris1999real}, who studied the case of even symmetric octics, was able to show that the cones of even symmetric octics that are sums of squares coincides with the non-negative cone. Recently, Goel, Kuhlmann and Reznick \cite{goel2017analogue} constructed even symmetric polynomials of every degree $2d>8$ and every number of variables $n>3$ which are non-negative but not a sum of squares, so for even symmetric forms Harris' example remains the only exceptional case compared to Hilbert's classification. Despite the classical case analysis done by Hilbert, it can also be interesting to study the quantitative comparison of sums of squares on non-negative polynomials in an  asymptotic situation, i.e., when the number of variables grows to infinity.  In contrary to the general situation, where for large numbers of variables almost every non-negative form is not a sum of squares (see \cite{blekherman2006there}) a detailed analysis of the symmetric sum of squares cone and symmetric non-negative cone in \cite{blekrie} showed that this is not the case in the symmetric case and that in particular in the quartic case the two cones coincide in the limit.

In this article, we study further the previously mentioned lines of research by focusing on the situation of sums of squares invariant under some families of finite real reflection groups $G\subset \Gl_n(\R)$. Such  groups  are generated by a set of orthogonal reflections across hyperplanes passing through the origin. The invariant theory of these groups is well understood and  generalizes the theory of symmetric polynomials. Therefore, our setup provides a natural unification and extension to the previously mentioned works on symmetric and even symmetric forms.  

\smallskip
\paragraph{\bf Outline of the article and contributions:} The beginning of the next section gives a short general introduction to the machinery of symmetry reduction for sums of squares  based on linear representation theory. In the case of finite reflection groups these techniques combined with results from invariant theory, and in particular the coinvariant algebra and harmonic polynomials,  allow for a concrete description of the qudratic module of invariant sums of squares in Theorem \ref{thm:sums of squares description}. The results we give in this second section are similar to previous works, notably \cite{blekrie,valentin,gatermann2004symmetry,vallentin2009symmetry}. 

Section 3 then turns to the special situation of the three infinite families $A_n$, $B_n$ and $D_n$ of irreducible reflection groups for which we can integrate the notion of the higher Specht polynomials \cite{ariki1997higher} with the previously  mentioned techniques. These polynomial allow for  a convenient way to combinatorially describe an isotypic decomposition  of the coinvariant algebra in the case of finite reflection groups whose irreducible components fall to the classes $A_{n},B_n,D_n$ (see Theorem \ref{thm:irreducible reps}). As we show in Theorem \ref{thm:general sums of squares repr} this combinatorial description then in turn implies a concrete characterization of  the cone of invariant sums of squares. In particular, we show in Theorem \ref{thm:ModulesStabilize} that if the degree $2d$ is fixed and the number of variables $n$ is growing, a stabilization of the isotypic decomposition and a resulting combinatorial stabilization of the structure of the cone of invariant sums of squares is happening in the case of all three families.   

Building on these general results, we study the cone of even symmetric (i.e., $B_n$-invariant) forms of degree 8 in more detail in section \ref{sec:even}. In Theorem \ref{Thm:DualEvenSymOct} we obtain an explicit description of the dual cone of even symmetric octics, which we can use to revisit the remarkable findings of Harris, which follow immediately from our description. Furthermore, we provide a complete description of the cone of even symmetric quartic sums of squares for all number of variables in Theorem \ref{thm:EvenSymSOSOctics}. Following our discussion of even symmetric forms, we turn to forms that are $D_n$-invariant in subsection \ref{secdn}. We first show that  Harris' remarkable equality for even symmetric ternary quartics  remains valid for forms invariant under the slightly smaller group $D_3$ (see Theorem \ref{thm:harrisd3}).  We then examine the dual cone of $D_4$ invariant quartic sums of squares in Theorem \ref{thm:D4simplicialCone}, which turns out to be simplicial. Similarly to our approach in the even-symmetric case, this yields in particular that every $D_4$-invariant quarternary quartic non-negative form is a sum of squares. These results allow to completely characterize the cases in which for $D_n$ invariant forms we have an equality between the cones of  sums of squares and non-negative forms (see Theorem \ref{thm:chardn}). To conclude our considerations, we highlight some connections to non-negativity testing of forms with the help of semidefinite programming in the last subsection. It follows from recent works of Scheiderer \cite{scheiderer} that the cone of non-negative forms in general is not a so called spectrahedral shadow, i.e., it can in general not be represented by projections of feasibility sets of semidefinite programming. In contrast to this result, we observe that additionally to the cases where the cone of invariant sums of squares coincides with the corresponding cone of non-negatives, there are cases where we can represent the cone of non-negative forms by linear matrix inequalities.
\section{Invariant sums of squares}

\subsection{General symmetry reduction}
Let $\underline{X}:=(X_1,\ldots,X_n)$ always denote a tuple of variables and write $\R[\underline{X}]:=\R[X_1,\ldots,X_n]=\bigoplus_{d \in \N_0}H_{n,d}$ for the polynomial ring in these variables, where $H_{n,d}$ denotes the subspace of forms of degree $d$. Let $G\subset \Gl_n(\R)$ be a finite group acting linearly on $\R^n$. This action then naturally gives rise to an action of $G$ on the polynomial ring $\R[X_1,\ldots X_n]$ and thus we can view this $\R$-vector space  as a $G$-module. It follows from Maschke's theorem that this $G$-module is completely reducible, and thus for any degree $d$ there exists an isotypic decomposition, i.e., the $G$-module $H_{n,d}$ decomposes into a direct sum of the form
\begin{eqnarray}\label{eq:decomp}
H_{n,d} \ = \ V^{(1)} \oplus V^{(2)} \oplus \cdots \oplus V^{(h)} \,
\end{eqnarray}
with $V^{(j)} = \theta^{(j)}_1 \oplus \cdots \oplus \theta^{(j)}_{\eta_j}$ and $\vartheta_j := \dim \theta^{(j)}_i$, where $\theta^{(u)}_{i_1}, \theta^{(v)}_{i_2}$ are $G$-isomorphic if and only if $u = v$ i.e., we denote by $\eta_j$ the \emph{multiplicity} of an irreducible $G$-module and by $\vartheta_j$ its \emph{dimension}. 
Here, the $\theta^{(j)}_i$ are the \emph{irreducible components} and the $V^{(j)}$ are the \emph{isotypic components}, i.e., the direct sum of isomorphic irreducible components. The component with respect to the trivial irreducible
representation in $\R[\underline{X}]$ is the invariant ring $\R[\underline{X}]^G$. In general, an irreducible representation $\theta_i^{(j)}$ will occur with infinite multiplicity in $\R[\underline{X}]$. Any irreducible representation $\theta$ occurs $\dim \theta$ many times in the regular representation $\R[G]$ of $G$, i.e., $\vartheta = \eta$ for a representation $\theta$ in $\R[G]$. For $f \in \R[\underline{X}]$ we write $\langle f \rangle_G$ for the $G$-module which is the linear span of $\{ \sigma f : \sigma \in G\}$.

It is classically known that $\R[\underline{X}]^G$ is a finitely generated $\R$-algebra, and furthermore each isotypic component in $\R[\underline{X}]$ is a finitely generated $\R[\underline{X}]^{G}$-module (see \cite[Theorem 1.3]{stanley}). These properties follow for finite groups from the existence of a linear projection onto $\R[\underline{X}]^G$, called the \emph{Reynolds-Operator}.
\begin{definition}
For a finite group $G$ the linear map $$\abb{\mathcal{R}_G}{H_{n,d}}{H_{n,d}^G}{f}{\frac{1}{|G|}\sum_{\sigma\in G}\sigma(f)}$$ is called the Reynolds operator of $G$. 
\end{definition}
\begin{remark}
Although we restrict to finite groups, most of the theory presented in this section can be directly translated  to the  more generally setup of  compact and reductive groups.  
\end{remark}

An important tool for the study of invariant sums of squares is Schur's lemma, which we include for the convenience of the reader.
\begin{lemma}[Schur's lemma] \label{le:schur}
Let $\mathbb{K}$ be a field which is algebraically closed and $V$ a $G$-module defined over $\mathbb{K}$. Further, let $\mathcal{V},\mathcal{W}$ denote two irreducible $G$-submodules of $V$. Then the $G$-module $\Hom_G(\mathcal{V},\mathcal{W})$ of $G$-homomorphisms between $\mathcal{V}$ and $\mathcal{W}$ satisfies    
\begin{align*}
    \Hom_G(\mathcal{V},\mathcal{W}) \cong \K
\end{align*}
if and only if $\mathcal{V}$ and $\mathcal{W}$ are $G$-isomorphic. Otherwise $\Hom_G(\mathcal{V},\mathcal{W})=0$.
\end{lemma}

\begin{remark}
In the sequel, we will mostly work with $G$-modules defined over the real numbers. In this setup, one devotes some care to the fact that irreducible representations defined over the reals may be reducible over the complex numbers. This additional difficulty is in fact not hard to overcome and, in particular, in the case of real reflection groups, which are the main focus of this work, all complexifications of real irreducible $G$-modules remain irreducible \cite{humphreys1990reflection}.
\end{remark}

 Let $\mathcal{V} = \langle f_1 \rangle_G$ be irreducible. As a consequence of Schur's lemma, we obtain that any $G$-homomorphism $\phi \in \Hom_G(\mathcal{V},\mathcal{W})$ is uniquely defined by $f_2 : = \phi (f_1).$ If further $\phi \neq 0$ then for any $\psi \in \Hom_G(\mathcal{V},\mathcal{W})$ it is $\psi = \lambda \phi, \lambda \in \K$. It motivates the following: 

\begin{definition}\label{def:SymmetryAdaptedbasis}
Let $V$ be a finite dimensional $G$-module with isotypic decomposition $$V=\bigoplus_{j=1}^l\bigoplus_{i=1}^{\eta_j} \theta^{(j)}_{i},$$ and $f_{ji} \in  \theta^{(j)}_{i}$ such that for every $j$ each $f_{ji}$ is  the image of one fixed $f_{j1}$ under a $G$-isomorphism (which is unique up to scalar multiplication). Then $\left( f_{11},\ldots,f_{1\eta_1},f_{21},\ldots,f_{l\eta_l}\right)$ is called a symmetry adapted basis of $V$. 
\end{definition}
We point out that while a symmetry adapted basis of a $G$-module is usually not a vector space basis, a basis is given by its $G$-orbit.

Note that an invariant polynomial which can be expressed as a sum of squares in the ring $\R[\underline{X}]$ will not necessarily have a sum of squares decomposition in invariant polynomials, i.e., 
$$\R[\underline{X}]^G\bigcap\sum \R[\underline{X}]^2\neq \sum(\R[\underline{X}]^G)^2.$$

By integrating the idea of a symmetry adapted basis together with Schur's lemma, one arrives at the following observation more or less directly (see also \cite{blekrie,cimprivc2009sums,gatermann2004symmetry,riener2013exploiting} for more details on  the following statement). \\ \\
For a $\R$-vector space $W$ we write $\sum W^2$ for the sums of squares of elements in $V$.
\begin{thm}\label{THM Decomp}
Let $\{f_{11},f_{12},\ldots,f_{l \eta_l}\}$ be a symmetry adapted basis for the $G$-module $H_{n,d}$ of forms of degree $d$. Then any $G$-invariant sum of squares form in $H_{n,2d}^G$ is contained in the set \begin{align*}
   \sum_{j=1}^l \mathcal{R}_G \left( \langle f_{j1},\ldots,f_{j\eta_j} \rangle_\R^2 \right)
    \end{align*}
\end{thm}

In some situations, it is convenient to formulate Theorem \ref{THM Decomp} in terms of matrix polynomials, i.e., matrices with polynomial entries. Given two $k\times k$ symmetric matrices $A$ and $B$ define their inner product as $\langle A,B \rangle=\operatorname{trace}(AB).$ 
We define a block-diagonal symmetric matrix $B$ with $j$ blocks $B^{(1)},\dots,B^{(j)}$ with the entries of each block given by:
\begin{equation}\label{eq:B}
B^{(j)} = \left( \mathcal{R}_G (f_{ju} \cdot f_{jv} \right))_{u,v} .
\end{equation}
Then Theorem \ref{THM Decomp} is equivalent to the following statement:
\begin{cor} \label{cor:sos}
Let $g \in H_{n,2d}^G$. Then $g \in \Sigma_{n,2d}^G$ if and only if \begin{align*}
    g = \langle A_1 \cdot B^{(1)} \rangle + \ldots +  \langle A_l \cdot B^{(l)} \rangle,
\end{align*}
for some $A_j \in \emph{Mat}_{\eta_j \times \eta_j}(\R)$ symmetric and positive semidefinite matrices.
\end{cor}

\subsection{Representation theory of finite reflection groups}
The aim of this subsection is to provide an introduction to the representation theory of finite real reflection groups and how their symmetry can be exploited to reduce complexity in calculations. The presented material is mainly based on work in \cite{blekrie, valentin, gatermann2004symmetry,riener2013exploiting}.
\begin{definition}
A real reflection group is a pair $(G,\rho)$, where $G$ is a finite group, $V$ a finite dimensional $\R$-vector space and $\rho : G \rightarrow \Gl (V)$ a linear representation of $G$ such that $\rho(G)$ is generated by a set of reflections. A reflection group is called essential, if the action of $G$ on $V$ does not contain a non-trivial $G$-submodule.
\end{definition}
Usually, we just say that a group $G$ is a reflection group and the relevant linear map $\rho$ should be understood from the context. An action of $G$ on $\R^n$ induces naturally an action on the polynomial ring in $n$ variables. 
\begin{example}
\begin{enumerate}
\item[(i)] The symmetric group $\mathfrak{S}_n$ on $n$ letters is a reflection group acting via coordinate permutation on $\R^n$. The action of $\mathfrak{S}_n$ on $\R^n$ is not essential, as the linear subspace $\R \cdot (1,\ldots,1)$ is fixed point wise. The induced action of $\mathfrak{S}_n$ on $\R^n / \R \cdot (1,\ldots,1)$ is known as the reflection group of type $A_{n-1}$ and is essential.
\item[(ii)]The symmetry group of the regular $m$-gon is a reflection group and called the dihedral group and denoted by $I_2(m)$.
\end{enumerate}
\end{example}
\begin{remark}
Any real reflection group can be identified with a direct product of essential reflection groups. The essential real reflection groups have been classified and are precisely the infinite series $A_{n-1},B_n,D_n,I_2(m)$ and the six exceptional reflection groups $E_6,E_7,E_8,F_4,H_3,H_4$ (see e.g., \cite{humphreys1990reflection}).
\end{remark}
The reflection group of type $B_n$ can be identified with the hyperoctahedral group $\mathfrak{S}_2 \wr \mathfrak{S}_n$ acting on $\R^n$ via sign changing and permutation of coordinates. Then $B_n$ is generated by the reflections at $\{X_i=\pm X_j\}$, for $1 \leq i \leq j \leq n$. Furthermore, $D_n$ can be identified with the subgroup of $B_n$ of index $2$, generated by the reflections at $\{X_i=\pm X_j\}$, for $1 \leq i < j \leq n$. 

\begin{thm}[Chevalley-Shephard-Todd theorem]\label{Thm:Chevalley-Shephard-Todd theorem}
Let $G$ be a finite group and let $G$ act linearly on $\R^n$. Then the invariant ring $\R[\underline{X}]^G$ is as $\R$-algebra isomorphic to a polynomial ring if and only if $G$ is a real reflection group. Moreover, in this case $\R[\underline{X}]^G$ is generated by $n$ algebraically independent forms $\psi_1,\ldots,\psi_n$, i.e., $$\R[\underline{X}]^G = \R[\psi_1,\ldots,\psi_n].$$ 
\end{thm}
While the generators are not unique but well explored (e.g., the elementary symmetric polynomials or the power sums are generators for the symmetric group), the multisets of their degrees $\{d_1,\ldots,d_n\}$ are unique and $\prod_i d_i = |G|$ (consult e.g., \cite{humphreys1990reflection} for further details). 
\begin{definition}\label{def:ng}
Let $G$ be a finite reflection group and $(d_1,\ldots,d_n)$ the sequence of degrees of the fundamental invariants. Then, we define
$$N_G(k):=|\{ (\alpha_1,\ldots,\alpha_n)\in\N_0^n\,:\, \alpha_1 d_1+\ldots+\alpha_n d_n=k\}|.$$
\end{definition}
With this definition the following is a direct consequence of Theorem \ref{Thm:Chevalley-Shephard-Todd theorem}.
\begin{cor}\label{cor:1}
Let $G$ be a finite reflection group. Then 
the dimension of the vector space of $G$-invariant forms of degree $d$ equals $N_G(d)$, i.e., $\dim H_{n,d}^G=N_G(d)$. 
\end{cor}

\begin{example}
\begin{itemize}
    \item[(i)] $\R[\underline{X}]^{\mathfrak{S}_n} = \R[e_1,e_2,\ldots,e_n]=\R[p_1,p_2,\ldots,p_n],$ where \\
    $e_j (\underline{X}) := \sum_{I \subset [n] : |I|=j}\prod_{i \in I} X_i$ are the elementary symmetric and $p_j(\underline{X}):=\sum_{i=1}^nX_i^j$ are the power sum polynomials.
    \item[(ii)] $\R[\underline{X}]^{B_n}=\R[e_1(\underline{X}^2),e_2(\underline{X}^2),\ldots,e_n(\underline{X}^2)]=\R[p_2,p_4,\ldots,p_{2n}],$ where $\underline{X}^2:=(X_1^2,\ldots,X_n^2)$.
    \item[(iii)]  $
     \R[\underline{X}]^{D_n} = \R[p_2,p_4,\ldots,p_{2n-2},e_n]$. \item[(iv)] $\R[\underline{X}]^{I_2(m)} = \R[X_1^2+X_2^2,  (X_1+\sqrt{-1}X_2)^m+(X_1-\sqrt{-1}X_2)^m]$.
\end{itemize}
\end{example}
\begin{remark}
For $\lambda := (\lambda_1,\ldots,\lambda_l) \in \N^l$ we often write $p_{\lambda} := p_{\lambda_1}\cdots p_{\lambda_l}$ for the $l$ products of the power sums $p_{\lambda_i}$.
\end{remark}
From a computational perspective, invariant theory as outlined above can be used to reduce computations for polynomials in $\R[\underline{X}]$ to the smaller ring $\R[\underline{X}]^G$.
Since $\R[\underline{X}]$ is in general a finite $\R[\underline{X}]^G$- module, the quadratic module $\R[\underline{X}]^G\bigcap\sum \R[\underline{X}]^2$ can be described quite conveniently. We outline this in the case of reflection groups below, using the coinvariant algebra and a theorem of Chevalley. 
\begin{definition}
Let $G$ be a reflection group acting linear on $\R^n$ and $\R[\underline{X}]^G = \R[\psi_1,\ldots,\psi_n]$. We call the forms $\psi_1,\ldots,\psi_n$ the fundamental invariants of $G$. The quotient $\R$-algebra of the polynomial ring modulo the ideal generated by the non-constant elements of the invariant ring is called the coinvariant algebra of $G$ and denoted by $\R[\underline{X}]_G$, i.e., $$\R[\underline{X}]_G := \R[\underline{X}] / \left(\psi_1,\ldots,\psi_n \right)_{\R[\underline{X}]}.$$
\end{definition}
The coinvariant algebra of $G$ has the structure of a $G$-module.

\begin{thm}\cite{lehrer2009unitary} \label{thm:coinvariant algebra}
Let $G$ be a real reflection group acting linear on $\R^n$. Then the coinvariant algebra $\R[\underline{X}]_G$ is as $G$-module isomorphic to the regular representation $\R[G]$ and $$\R[\underline{X}] \cong \R[\underline{X}]^G \otimes_\R \R[\underline{X}]_G$$ as graded $\R$-algebras.  
\end{thm}

\begin{cor} \label{cor:deomposition of polynomial ring}
Let $\R[\underline{X}]^G= \R[\psi_1,\ldots,\psi_n]$ be a polynomial ring in the fundamental invariants $\psi_1,\ldots,\psi_n$. Let $\R[\underline{X}]_G = \bigoplus_{j=1}^l \eta_j \theta^{(j)}$ be the isotypic decomposition of the coinvariant algebra. Then there exists a symmetry adapted basis $f_{11},\ldots,f_{l\eta_l} \in \R[\underline{X}]$ of $\R[\underline{X}]_G$ such that any $f \in \R[\underline{X}]$ can be written as $$f =  \sum_{j=1}^l \sum_{i=1}^{\eta_j} \sum_{\sigma \in G} g_{ji,\sigma} \sigma f_{ji}, $$ where $g_{ji,\sigma} \in \R[\underline{X}]^G$.
\end{cor}
\begin{proof}
The existence of the symmetry adapted basis $(f_{11},\ldots,f_{l\eta_l})$ of $\R[\underline{X}]_G$ follows by Schur's lemma \ref{le:schur}. Further, by definition, the $G$-orbit of $(f_{11},\ldots,f_{l\eta_l})$ is a vector space basis of the coinvariant algebra. The claim follows from the graded tensor decomposition by Theorem \ref{thm:coinvariant algebra}.
\end{proof}
The second sum in the representation of a polynomial in Corollary \ref{cor:deomposition of polynomial ring} goes up to $\eta_j$. We recall out that the multiplicity $\eta_j$ of an irreducible representation $\theta^{(j)}$ in the coinvariant algebra equals the dimension $\vartheta_j$.
\begin{remark}
The calculation of one symmetry adapted basis of the coinvariant algebra allows easily the computation of the isotypic composition of the $G$-module $H_{n,d}$ for any degree. As a rough general procedure, one needs to compute the products of elements from the symmetry adapted basis with fundamental invariants of $G$, such that the degree of the obtained homogeneous polynomial equals $d$.
\end{remark}
\begin{definition}
Let $S:=\{s_1,\ldots,s_{|G|}\}$ be a basis of $\R[\underline{X}]_G$. Then we define the matrix polynomial $H^S(\psi_1,\ldots,\psi_n)\in\R[z]^{|G| \times |G|}$ to be
$$H^S_{u,v}:=R_G(s_u\cdot s_v),$$
where we express each entry $R_G(s_u\cdot s_v)$ in terms of the fundamental invariants $\psi_1,\ldots,\psi_n$.
\end{definition}

\begin{lemma}
Let $f\in\R[\underline{X}]$ be $G$-invariant and let $\gamma\in\R[\psi_1,\ldots,\psi_n]$ with $\gamma(\psi_1,\ldots,\psi_n)=f$ then $f$ is a sum of squares if and only if $\gamma(\psi_1,\ldots,\psi_n)$ admits a representation of the form
$$\gamma= \Tr (G\cdot H^S),$$
where $G$ is a sum of squares matrix polynomial, i.e., $G=L^tL$ for some $L(\psi_1,\ldots,\psi_n)\in\R[\psi_1,\ldots,\psi_n]^{n\times m}$ for some $1\leq m\leq n.$
\end{lemma}
\begin{proof}
This follows from the decomposition $\R[\underline{X}]\cong \R[\underline{X}]^G \otimes \R[\underline{X}]_G$ in Theorem \ref{thm:coinvariant algebra}.
\end{proof}

Working with a symmetry adapted basis allows the following 

\begin{definition}
For every irreducible representation $\theta^{(j)}$ of $G$ we can construct a matrix polynomial $H^{\vartheta_j}\in\R[\psi_1,\ldots,\psi_n]^{\eta_j\times\eta_j}$ in the following way: 
Let $\R[\underline{X}]_G=\bigoplus_{i=1}^l \R[\underline{X}]_G^{\vartheta_j}$ be the isotypic decomoposition of the coinvariant algebra and further $\{s_{1,1},\ldots,s_{1,\eta_1},s_{2,1},\ldots,s_{l,\eta_l}\}$ be a symmetry adapted basis of $\R[\underline{X}]_G$.
Then we define $$H^{\vartheta_j}_{u,v}=R_G(s_{j,u}\cdot s_{j,v}).$$
\end{definition}
Combining above definition and lemma, and the results from Schur's lemma we immediately get
\begin{thm}\label{thm:sums of squares description}
Let $G$ be a finite reflection group with $\R[\underline{X}]^G =\R[\psi_1,\ldots,\psi_n]$, then we have
$$\Sigma \R[\underline{X}]^2 \cap \R[\underline{X}]^G= \left\{ g\in \R[\psi_1,\ldots,\psi_n]\,:\, g=\sum_{j=1}^l \Tr(H^{\vartheta_j}\cdot A_j)\right\},$$
where $A_j\in{ \R[\psi_1,\ldots,\psi_n]}^{\eta_j\times\eta_j}$ is a sum of squares matrix polynomial.
\end{thm}

\begin{example}
Let $f\in\R[X_1,X_2]$ be a homogeneous polynomial of degree $2d$ which is invariant under a dihedral group $I_2{(k)}$. The dihedral group $I_2{(k)}$ has only irreducible representations of dimension $1$ or $2$. In fact, if $k$ is odd (resp. even), then $2$ (resp. $4$) representations of dimension one and $\frac{k-1}{2}$ (resp. $\frac{k-2}{2}$) representations of dimension two. By block-diagonalisation we end up with $H^S(z)$ and $G(z)$ having $2$ (resp. 4) $1 \times 1$ blocks $H^{\theta_1},H^{\theta_2}$ (resp. $H^{\theta_1},\ldots,H^{\theta_4}$) and $\frac{k-1}{2}$ (resp. $\frac{k-2}{2}$) $2 \times 2$ blocks $H^{\theta_3},\ldots,H^{\theta_{\frac{k+3}{2}}}$ (resp. $H^{\theta_5},\ldots,H^{\theta_\frac{k+6}{2}}$). Then for $n$ odd (resp. even) $f\geq 0$ if and only if there exist sums of squares matrix polynomials $A_i\in{ \R[X_1^2+X_2^2, (X_1+\sqrt{-1}X_2)^k+(X_1-\sqrt{-1}X_2)^k]}^{\dim \theta_j\times\dim\theta_j}$  
such that \begin{align*}
    f & = \sum_{j=1}^{m} \Tr \left( H^{\theta_j} \cdot A_j \right),
\end{align*}
where $m = \frac{k+3}{2}$ (resp. $m =\frac{k+6}{2}$).

For $k = 3$ the coinvariant algebra $\R[x,y]_{I_{2}(3)}$ decomposes into $$\theta^{(1)} = \langle 1 \rangle, \theta^{(2)} = \langle -x^3+3xy^2 \rangle, \theta^{(3)}_1 = \langle x,y\rangle, \theta^{(3)}_2 = \langle xy,x^2-y^2 \rangle,$$
where $\theta^{(3)}_1$ and $\theta^{(3)}_2$ are $I_{2}(3)$-isomorphic via $x \mapsto xy$. Then $$H^{\theta^{(1)}} = (1), H^{\theta^{(2)}} = \left( \mathcal{R}_{I_2(3)} (3xy^2-x^3)^2\right), H^{\theta^{(3)}} = \left(\begin{array}{cc}
 \mathcal{R}_{I_2(3)} (x^2)    & \mathcal{R}_{I_2(3)}(x^2y) \\
  \mathcal{R}_{I_2(3)} (x^2y)   & \mathcal{R}_{I_2(3)} (x^2y^2) 
\end{array}\right). $$
\end{example}

\begin{definition}
Let $G$ be a finite reflection group and $\theta$ an irreducible representation. We write $h_k^{\vartheta}$ for the multiplicity of $\theta$ in $(\R[\underline{X}]_G^{\theta})_k$, i.e., the multiplicity of $\theta$ in the isotypic decomposition of the subspace of forms of degree $k$ in the coinvariant algebra.
\end{definition}
We recall that $N_G(d)$ denotes the vector space dimension of $G$-invariant forms of degree $d$  \ref{cor:1}.
\begin{cor}\label{cor:bound}
Let $G$ be a finite reflection group and $\theta$ be an irreducible representation. Then the multiplicity of the corresponding irreducible representation in the $G$-module $H_{n,d}$ equals  
$$\sum_{k=0}^dN_G(d-k)\cdot h^{\vartheta}_k.$$
\end{cor}

\subsection{G-harmonic polynomials}
In this subsection we present a specific basis of the coinvariant algebra for reflection groups which can be simply computed.

\begin{definition}
For a polynomial $f=\sum_\alpha c_\alpha \underline{X}^\alpha  \in \R[\underline{X}]$ we denote by $f(\partial)$ the linear operator $$ \abb{f(\partial )}{\R[\underline{X}]}{\R[\underline{X}]}{g}{\sum_\alpha c_\alpha \frac{\partial^\alpha}{ (\partial \underline{X})^\alpha }g,}$$ i.e., $f(\partial)$ is the formal sum of scaled partial derivatives considered as a linear map. 
\begin{example}
Let $f = X_1^2+X_1X_2 \in \R[X_1,X_2,X_3]$, then $f(\partial) = \frac{\partial^2}{\partial X_1 \partial X_1} + \frac{\partial^2}{\partial X_1 \partial X_2}$ and $$f(\partial) \left( X_1^2+X_2^2+X_3^2 +X_1X_2X_3 \right) = 1 +X_3.$$
\end{example}
\end{definition}
\begin{definition}
Let $G$ be a real reflection group and $\R[\underline{X}]^G = \R[\psi_1,\psi_2,\ldots,\psi_n]$. We define the $\R$-vector space of harmonic polynomials $\mathcal{H}_G := \left( \R[\underline{X}]^G\right)^\perp$, with respect to the scalar product on $\R[\underline{X}]$ given by $$ \abb{\langle \cdot , \cdot \rangle}{\R[\underline{X}] \times \R[\underline{X}]}{\R[\underline{X}]}{(f,g)}{ \ev_{(0,\ldots,0)}\left( f(\partial) g(\underline{X}) \right).}$$
\end{definition}
\begin{thm}\cite{bergeron2009algebraic}
Let $G$ be a real reflection group and $\Delta:=\prod L_i$, be the product of the linear polynomials defining the reflection hyperplanes. Then, the vector space of $G$-harmonic polynomials $\mathcal{H}_G$ is generated by all partial derivatives of $\Delta$, i.e., $\mathcal{H}_G = \langle \frac{\partial^\alpha}{\partial x^\alpha}\Delta : \alpha \in \N_0^n \rangle_\R$. Furthermore, $\mathcal{H}_G$ is as $G$-module isomorphic to the regular representation of $G$ and $\R[\underline{X}] = \R[\psi_1,\ldots,\psi_n]\otimes_\R \mathcal{H}_G.$ 
\end{thm}

\begin{remark}
Let $G$ be a finite reflection group and $\psi_1,\ldots,\psi_n$ generators of the invariant ring.  Consider the map $$\abb{\Psi}{\R^n}{\R^n}{\underline{X}}{(\psi_1(\underline{X}),\ldots,\psi_n(\underline{X})).}$$
Then, thanks to a statement of Steinberg in \cite{steinberg1960invariants} we have $$\Delta = c\cdot \emph{Jac } \Psi,$$ where $c \in \R \setminus \{0\}$ and $\emph{Jac } \Psi$ denotes the Jacobian matrix of $\Psi$. The choice of fundamental invariants $\psi_1,\ldots,\psi_n$ does not matter. 
\end{remark}
\begin{example}
For $\mathfrak{S}_n$ the symmetric group acting on $\R^n$ via coordinate permutation and $\psi_i := p_i = \sum_{j=1}^n X_j^i$ the power sums, we obtain $\Delta = \prod_{i < j}(x_i-x_j)$ equals the determinant of the Vandermonde matrix. $\Delta$ is the Jacobian of $\Psi$, which is precisely the product over all reflections $\{ X_i = X_j\}$ of $\mathfrak{S}_n$.
\end{example}

\begin{remark}
Computing a basis of the coinvariant algebra $\R[\underline{X}]_G = \R[\underline{X}] / \R[\underline{X}]^G_{> 0}$, that is defined as a quotient space, is highly complex and involves the calculation of a Groebner basis. However, the approach using harmonic polynomials is more efficient because it is based on linear algebra for given fundamental invariants. As the fundamental invariants of real reflection groups are well-known, one can calculate the polynomial $\Delta$ and all its partial derivatives explicitly. 
\end{remark}

\subsection{Convex geometric properties of $\Sigma^G,\mathcal{P}^G$}
An interesting and highly useful feature of $\Sigma_{n,2d}^G$ and $\mathcal{P}_{n,2d}^G$ is their \emph{convex geometry}, which enables the use of convex geometric techniques to study these sets. In the research on non-negativity versus sums of squares have the convex cones and their dual cones been studied intensively (see e.g., Blekherman's work in \cite{blekherman2012nonnegative} on Hilbert's inequality cases or \cite{blekherman2012semidefinite}). In this subsection, we present known and adapted knowledge on the convex geometric properties of $\Sigma_{n,2d}^G$ and $\mathcal{P}_{n,2d}^G$. We refer to subsection $4.5$ in \cite{blekrie} for more details.

\begin{remark}
\begin{itemize}
    \item $\Sigma_{n,2d}^G$ and $\mathcal{P}_{n,2d}^G$ are \emph{convex cones}, i.e., they are convex sets which are closed under scalar multiplication by non-negative scalars. Moreover, these sets are closed and \emph{pointed} (i.e., they do not contain a non-trivial linear subspace), see e.g., \cite{blekherman2006there}. Such convex cones are called \emph{proper}.
    \item Let $K\subset\R^N$ be a proper convex cone. The dual cone $K^\ast$ is defined as $$K^\ast := \left\{ \ell \in \Hom \left( \R^N,\R\right) : \ell(P)\subseteq \R_{\geq 0}\right\}.$$
    \item We associate a linear functional $\ell \in \left( H_{n,2d}^{G}\right)^\ast$ with a quadratic form $Q_{\ell}$ defined by $$\abb{Q_\ell}{H_{n,d}}{\R}{f}{\ell \left( \mathcal{R}_G(f^2)\right).}$$
\end{itemize}
\end{remark}

Since we are in the homogeneous case we have the following description of the dual cone of invariant non-negative forms:
\begin{prop}\label{pro:DualOfNonneg} \cite{blekherman2006there}
The dual cone of the non-negative invariant forms is the convex cone that is generated by all point-evaluations, i.e., $\left( \mathcal{P}_{n,2d}^G \right)^\ast  = \cone\{\ev_a : a \in \mathbb{S}^{n-1} \}$   where $$\abb{\ev_a}{\R[\underline{X}]}{\R}{f(\underline{X})}{f(a).}$$  
 \end{prop}
By duality any $f \in \mathcal{P}_{n,2d}^G$ contained in the boundary has a real projective zero.

We formulate the dual version of Theorem \ref{THM Decomp}.
\begin{lemma}\label{le:psd}\label{cor:DualSosCone}
Let $\ell \in \left(H_{n,2d}^G\right)^\ast$ and $\{f_{11},\ldots,f_{1\eta_1},f_{21},\ldots,f_{l \eta_l}\}$ be a symmetry adapted basis for the space $H_{n,d}$ of forms of degree $d$ and $B^{(j)} = \left( \mathcal{R}_G (f_{ju}\cdot f_{jv} \right))_{u,v}$. Then $\ell$ is contained in $\left( \Sigma_{n,2d}^G\right)^\ast$ if and only if $\ell (B_j)$ is positive semidefinite for all $j=1,\ldots,l$.
\end{lemma}

The following lemma enables the characterisation of extremal elements in via their kernels.
\begin{lemma} \cite[Lemma 2.2]{blekherman2012nonnegative} \label{le:MaxOfKernel}
Let $V$ be a $\R$-vector space, $\mathcal{A}$ the vector space of quadratic forms on $V$ and $\mathcal{A}^+\subset \mathcal{A}$ the cone of positive semidefinite quadratic forms. Let $L$ be a linear subspace of $\mathcal{A}$ and $K$ be the section of $\mathcal{A}^+$ with $L$, i.e., $K:=\mathcal{A}^+\cap L.$
Then a quadratic form $Q \in K$ spans an extreme ray of $K$ if and only if its kernel is maximal for all forms in $L$, i.e., if $\ker Q\subseteq \ker P$ for a $P \in L$, it is
$P=\lambda Q$ for some $\lambda\in\R$.
\end{lemma}
In order to examine the kernels of quadratic forms, we use the following construction. For a linear subspace $W\subset H_{n,d}$, we define its quadratic symmetrization w.r.t. $G$ as 
$$W^{<2>}:=\left\{ h\in H_{n,2d}^G\,:\;h=\mathcal{R}_G \left(\sum f_i g_i\right)\;\text{ for } f_i \in W \text{ and  }g_i \in H_{n,d}\right\}.$$

In order to characterize the extreme rays of $\left( \Sigma_{n,2d}^G\right)^\ast$ we use Lemma \ref{le:psd} to identify the dual cone $\left( \Sigma_{n,2d}^G\right)^\ast$ with a linear section of the cone of positive semidefinite forms with the subspace of $G$-invariant quadratic forms on $H_{n,d}$.

\begin{prop}\cite{blekrie} \label{prop:W2}
An element $\ell \in \left( \Sigma_{n,2d}^G \right)^\ast$ is extremal if and only if $\ker Q_\ell$ is maximal among all kernels in $\left( \Sigma_{n,2d}^G \right)^\ast$. Let $W := \ker {Q}_\ell$, then $W^{<2>}$ is equal to the kernel of $\ell$. Moreover, if $\left(f_{11},\ldots,f_{l\eta_l}\right)$ is a symmetry adapted basis of $H_{n,d}$ and $\left(g_{11},\ldots,g_{l\eta^\prime_l}\right)$ is a symmetry adapted basis of $W$ such that $g_{ji_1}$ and $f_{ji_2}$ span $G$-isomorphic irreducible $G$-modules, and $f_{ji_2} \mapsto g_{ji_1}$ defines the unique $G$-isomorphism, then $$W^{\langle 2 \rangle} = \langle \mathcal{R}_G( g_{ji_1}\cdot f_{ji_2}) : 1 \leq j \leq l, 1 \leq i_2\leq \eta_j, 1\leq i_1 \leq \eta_j^\prime \rangle_\R.$$
\end{prop}
\begin{proof}
The first claim follows from Lemma \ref{le:MaxOfKernel}.
The second claim follows from the positive semidefiniteness of the quadratic form $Q_\ell$. The complexity reduction gives the above description of $W^{\langle 2 \rangle}$ according to the use of a symmetry adapted basis and applying Schur's lemma. 
\end{proof}
To prove equality or inequality of $\Sigma_{n,2d}^G$ and $\mathcal{P}_{n,2d}^G$ we propose a dual approach. By Minkowski's theorem, any element in a proper convex cone can be written as a conic combination of extremal elements.
\begin{cor}\label{cor:Sigma=P}
The sets of $G$-invariant $n$-ary non-negative and sums of squares forms of degree $2d$ are equal if and only if any extremal ray in $\left( \Sigma_{n,2d}^G\right)^\ast$ is generated by a point-evaluation. 
\end{cor}
\begin{proof}
The primal cones $\mathcal{P}_{n,2d}^G$ and $\Sigma_{n,2d}^G$ are equal if and only if the dual cones are equal. By Minkowski's theorem, any $ \ell \in \left( \Sigma_{n,2d}^G \right)^\ast$ can be written as a sum of extremal elements. If any extremal ray in $\left( \Sigma_{n,2d}^G\right)^\ast$ is generated by a point-evaluation, then there exists a set $M \subset \R^n$ such that $$ \left( \mathcal{P}_{n,2d}^G \right)^\ast \subseteq \left( \Sigma_{n,2d}^G \right)^\ast= \mbox{cone}\{\ev_{a} : a \in M \subset \R^n\} \subset \mbox{cone}\{ \ev_a : a \in \mathbb{S}^{n-1}\} = \left( \mathcal{P}_{n,2d}^G \right)^\ast,$$ where the last equality follows by Proposition \ref{pro:DualOfNonneg}. Conversely, if $\Sigma_{n,2d}^G = \mathcal{P}_{n,2d}^G$ then also the dual cones are equal. However, $\left( \mathcal{P}_{n,2d}^G \right)^\ast$ is the convex cone that is generated by all point-evaluations. Hence, any extremal ray in $\left( \Sigma_{n,2d}^G\right)^\ast$ is generated by a point-evaluation.
\end{proof}

\section{Sums of squares invariant under $A_{n-1}, B_n$, and $D_n$}
In this section we present an algorithmic approach for calculating a symmetry adapted basis of the coinvariant algebra for reflection groups of type $A_{n-1},B_n$ or $D_n$. This was introduced by the authors in \cite{ariki1997higher,morita1998higher}. Then we prove a stabilization of the isotypic decomposition for fixed degree and large enough number of variables for the series of essential reflection groups.

\subsection{Higher Specht polynomials}
A well known classical construction of the irreducible $\mathfrak{S}_n$-modules in the real polynomial ring is due to Specht \cite{specht}. The $\mathfrak{S}_n$-generators of these representations are called \emph{Specht polynomials}.  However, we are interested in the decomposition of the coinvariant algebra. An elegant combinatorial algorithm to decompose the coinvariant algebra into irreducibles for all complex reflection groups of type $G(m,n,p)$ was introduced in \cite{morita1998higher}. In the following, we briefly present their work. Furthermore, we present a combinatorial description of the invariant sums of squares forms.

We begin by recalling some basic definitions from combinatorics. 
\begin{definition}
A non-increasing sequence of positive integers $\lambda = (\lambda_1,\ldots,\lambda_l)$ is called a partition and $l$ is the length of $\lambda$. We denote by $|\lambda| = \sum_{i=1}^l \lambda_i$ the value of $\lambda$ and, say that $\lambda$ is a partition of $n$ if $|\lambda| = n$ and write $\lambda \vdash n$. For partitions $\lambda^1$ and $\lambda^2$ we call the pair $\Lambda = (\lambda^1,\lambda^2)$ a bipartition (here we also allow that either $\lambda^1 = \emptyset $ or $\lambda^2 = \emptyset$). We say that $|\Lambda| = |\lambda^1|+|\lambda^2| =n$ is the length of $\Lambda$ and write $\Lambda \vdash n$ for $\Lambda$ a bipartition of $n$.  
\end{definition}
We always denote bipartitions by capital letters and partitions by small letters. However, sometimes we write $(\lambda,\emptyset)$ instead of $\lambda$ for a partition $\lambda$. 
\begin{definition} 
The Young diagram associated to a partition $\lambda \vdash n$ is a sequence of ordered boxes starting from the left which $i$-th line contains $\lambda_i$ boxes. If one fills the boxes with all the integers in $[n]$, one calls the obtained object a Young tableau (or just tableau) of shape $\lambda$. If the numbers in all columns and rows are increasing we call it a standard tableau. \\
Bipartitions are associated with their pairs of Young diagrams. A Young bitableau (or just bitableau) is a filling of both Young diagrams with all the numbers in $[n]$ and we call it standard if both Young diagrams are standard. \\ 
We denote by $\YT (\Lambda)$ the set of (bi-)tableaux of shape $\Lambda$ and by $\SYT (\Lambda)$ the subset of standard (bi-)tableaux. 
\end{definition}
The famous Robinson-Schensted correspondence gives a bijection between the standard tableaux of shape $\lambda$ and the elements in the conjugacy class of $\mathfrak{S}_n$ which are labelled by $\lambda$. Hence, this number equals the multiplicity of the Specht module $S^{\lambda}$ in the coinvariant algebra. The correspondence has been adapted to complex reflection groups of type $G(m,n,p)$ and in particular for the contained series of reflection groups of types $B_n=G(2,1,n)$ and $D_n=G(2,2,n)$ (see e.g., section 10 in \cite{caselli2011projective}). \\

Following \cite{ariki1997higher} we introduce the objects arising in their construction of a symmetry adapted basis of the coinvariant algebra. The group $\mathfrak{S}_n$ acts naturally on a Young tableau by replacing the entry $i$ with $\sigma (i)$ for an element $\sigma \in \mathfrak{S}_n$.

\begin{definition}
 Let $T$ be a Young tableau of shape $\lambda \vdash n$. The $\mathfrak{S}_n$-subgroups \begin{align*}
    \mathcal{C}_T & := \left\{ \sigma \in \mathfrak{S}_n : \sigma T \emph{ is obtained by permutation of the columns of } T\right\}  \\
    \mathcal{R}_T & := \left\{ \sigma \in \mathfrak{S}_n : \sigma T \emph{ is obtained by permutation of the rows of } T\right\}
\end{align*}are called the column, resp. the row stabilizer of $T$. We define the formal linear combination $$\epsilon_T := \frac{f^\lambda}{n!}\sum_{\sigma \in \mathcal{C}_T, \tau \in \mathcal{R}_T} \emph{sgn} (\sigma) \sigma \tau \in \R[\mathfrak{S}_n],$$ where $f^\lambda$ is the number of standard tableau of shape $\lambda$. For a bitableau $T = (T^1,T^2)$ we define $\epsilon_{T^1}, \epsilon_{T^2} \in \R[\mathfrak{S}_n]$ analogously and set $\epsilon_T := \epsilon_{T^1}\cdot \epsilon_{T^2}$.
\end{definition}

We associate (pairs of) tableau with sequences, monomials and polynomials:
 \begin{definition}
 Let $T \in \YT (\Lambda)$ a (bi-)tableau. Then we define the word of $T$ as the sequence $w(T) \in \N^
{|\lambda|}$ where we read and notate each column of the tableau $T^1$ from the bottom to the top, starting from the left. We continue with this procedure for the tableau $T^2$. 

We define the index $i(T)$ of $T$ as follows. The number $1$ in the word $w(T)$ has index $0$. If $k$ in the word has index $p$, then $k+1$ has index $p$ or $p+1$ according as it lies to the right or the left of $k$.  We call the sum of the entries of $i(T)$ the charge of $T$ and write $\ch (T)$.

We associate to a tuple of (bi-)tableau $(T,S)$ of the same shape $\Lambda \vdash n$ a monomial in $n$ variables $\underline{X}_T^S := X_{w(T)_1}^{i(w(S))_1} \cdots X_{w(T)_{|\Lambda|}}^{i(w(S))_{|\Lambda|}}.$ Moreover, we define the polynomials associated to the pair $(T,S)$ $$F_{T}^S := \epsilon_T \cdot \underline{X}_{T}^{S} \in \R[\underline{X}] \mbox{ and } \widehat{F}_T^S:= F_{T}^S(\underline{X}^2) \cdot \prod_{j \in T^2} X_j,$$ where $\underline{X}^2:=(X_1^2,\ldots,X_n^2)$.
\end{definition}

\begin{example}
Let $\Lambda = ((2,1),(1)) \vdash 4$ be a bipartition and $S= \left( \ytableausetup{smalltableaux}
  \ytableaushort{
   14,2}\; ,
  \ytableaushort{
   3} \right), T=\left( \ytableausetup{smalltableaux}
  \ytableaushort{
   12,4}\; ,
  \ytableaushort{
   3} \right) \in \SYT (\Lambda) $.
The word of $S$ is $w(S)=(2,1,4,3)$ and the word of $T$ is $w(T) = (4,1,2,3)$. We calculate the indices $i(S)=(1,0,2,1)$ and $i(T)=(1,0,0,0)$ and compute $\underline{X}_T^S=X_4^1X_1^0X_2^2X_3^1=X_2^2X_3X_4$, $F_T^S=X_1^2X_3X_4+X_2^2X_3X_4-X_1X_2^2X_3-X_1X_3X_4^2$.
\end{example}
The authors in \cite{morita1998higher} used the following definition, referring to Specht's polynomial representation of the irreducible $\mathfrak{S}_n$-modules:
\begin{definition}
Let $G \in \{ A_{n-1},B_n,D_n\}$. We call the $G$-generators of the coinvariant algebra $\R[\underline{X}]_G$ higher Specht polynomials. 
\end{definition}
In the following, we will denote an irreducible representation labelled by a (bi-)partition $\Lambda$ by $S^\Lambda$. The underlying group should be clear from the context.

\begin{thm}\label{thm:irreducible reps}\cite{morita1998higher}
For reflection groups of type $A_{n-1},B_n$ or $D_n$ the higher Specht polynomials can be calculated as follows: 
\begin{enumerate}[a)]
\item For $A_{n-1}$ with $\lambda \vdash n$ and $B_n$  with $\Lambda \vdash n$  the higher Specht polynomials are given by the sets of polynomials $\left\{ F_T^S : (T,S) \right\}$ and $\left\{ \widehat{F}_T^S : (T,S)\right\}$ where $(T,S)$ varies over the set of all standard (bi-)tableaux of shape $\lambda$, resp. $\Lambda$.

\item  Let $\mathcal{L} := \{ \Lambda =(\lambda,\mu) \vdash n : \lambda \neq \mu, |\lambda|\geq |\mu|\}$. For $D_n$ the higher Specht polynomials are given as the union of the two sets
 \begin{align*}
     & \left\{ \widehat{F}_T^S : \,  \Lambda \in \mathcal{L}, (T,S) \in \SYT (\Lambda) \times \SYT (\Lambda) \right\} \,\text{, and} \\ 
   &  \left\{ \widehat{F}_{(T^1,T^2)}^S \pm \widehat{F}_{(T^2,T^1)}^S : (\lambda,\lambda) \vdash n, ((T^1,T^2),S) \in \SYT ((\lambda,\lambda)) \times \SYT ((\lambda,\lambda)) \right\}.
 \end{align*}
 Furthermore, for $(\lambda,\mu)$ and $(\mu,\lambda)$ the associated irreducible $B_n$-representations remain $D_n$-irreducible, but are $D_n$-isomorphic. For a pair $((T^1,T^2),S)$ of standard bitableaux of shape $ (\lambda,\lambda) \vdash n$ it is \begin{align*}
        \langle \widehat{F}_{(T^1,T^2)}^S \rangle_{D_n} = \langle \widehat{F}_{(T^1,T^2)}^S + \widehat{F}_{(T^2,T^1)}^S \rangle_{D_n} \oplus \langle \widehat{F}_T^S - \widehat{F}_{(T^2,T^1)}^S \rangle_{D_n},
    \end{align*} where the $D_n$-modules $ \langle \widehat{F}_{(T^1,T^2)}^S + \widehat{F}_{(T^2,T^1)}^S \rangle_{D_n}$ and $\langle \widehat{F}_{(T^1,T^2)}^S - \widehat{F}_{(T^2,T^1)}^S \rangle_{D_n}$ are irreducible and non-isomorphic. 
    \end{enumerate}
\end{thm}
    
Moreover, we find the following as a consequence of Schur's lemma \ref{le:schur} and the statements in \cite{morita1998higher}: For the groups $B_n$ and $D_n$ (resp. $A_{n-1}$) and $T=(T^1,T^2),S_1,S_2$ standard (bi-)tableaux of shape $\Lambda$ (resp. $\lambda$) the maps $$ \widehat{F}_T^{S_1}\mapsto \widehat{F}_T^{S_2} \text{ (resp. } {F}_T^{S_1}\mapsto F_T^{S_2})$$ define the (up to scalar) unique $G$-module isomorphism. In the case that $\Lambda$ has the form $(\lambda, \lambda)$, the unique $D_n$-isomorphism is then $$\widehat{F}_{(T^1,T^2)}^{S_1}\pm \widehat{F}_{(T^2,T^1)}^{S_1} \mapsto \widehat{F}_{(T^1,T^2)}^{S_2}\pm\widehat{F}_{(T^2,T^1)}^{S_2}.$$

\begin{definition}
Let $G \in \{A_{n-1},B_n,D_n\}$ and $\Lambda \vdash n$ a (bi-)partition. We write $q_d^\Lambda$ for the multiplicity of the $G$-module $S^\Lambda$ in $H_{n,d}$.
\end{definition}

\begin{remark} \label{rmk:deg of higher specht}
From Theorem \ref{thm:irreducible reps} we obtain a combinatorial description of $h_k^\theta$, i.e., of the multiplicity of an irreducible representation $\theta$ in the subspace of the coinvariant algebra of forms of degree $k$. Namely, in the case of $A_{n-1}$ $\theta$ is labelled by a partition $\lambda \vdash n$ and $$h_k^{\lambda} = |\{ T \in \SYT (\lambda) : \ch (T) = k \}|. $$ While for $B_n$ and $D_n$ $\theta$ is labelled by a bipartition $\Lambda = (\lambda,\mu) \vdash n$ and $$h_k^\Lambda = |\{ (T,S) \in \SYT (\Lambda) : 2\ch (T,S) + |\mu|=k\}|.$$
In particular, the multiplicity of $S^{\Lambda}$ in $H_{n,d}$ can be described combinatorially via the number of standard (bi-)tableaux and the degrees of $G$ $$q_d^\Lambda = \sum_{k=0}^dN_G(d-k) \cdot h_k^\Lambda.$$
\end{remark}

By integrating the above presented construction with the general setup, the degrees of the considered reflection groups and the standard (bi-)tableaux combinatorially encode the following information about the invariant  sums  of squares.
\begin{thm}\label{thm:general sums of squares repr}
Let $G \in \{ A_{n-1},B_n\}$.
\begin{enumerate}
    \item The isotypic decomposition of $H_{n,d}$ is $$\bigoplus_{\Lambda \vdash n}q_d^\Lambda \cdot S^{\Lambda},$$ where $\Lambda$ ranges over partitions for $A_{n-1}$ and otherwise bipartitions. 
    
\item There exists a symmetry adapted basis of the coinvariant algebra $\R[\underline{X}]_G$ consisting of higher Specht polynomials $(s_{1}^\Lambda,\ldots,s_{\vartheta_\Lambda}^{\Lambda})_{\Lambda \vdash n}$, where $\vartheta_\Lambda$ denotes the dimension of $S^{\Lambda}$. By defining symmetric matrix polynomials $H^\Lambda \in \R[\underline{X}]^{\vartheta_\Lambda \times \vartheta_\Lambda}$ via, $H^\Lambda_{v,u} := \mathcal{R}_G(s_v^\Lambda \cdot s_u^\Lambda)$ we have
$$\Sigma \R[\underline{X}]^2 \cap \R[\underline{X}]^G= \left\{ g\in \R[\psi_1,\ldots,\psi_n]\,:\, g=\sum_{\Lambda \vdash n} \Tr(H^{\vartheta_j}\cdot A_\Lambda)\right\},$$
where $A_\Lambda \in{ \R[\psi_1,\ldots,\psi_n]}^{\vartheta_\Lambda \times \vartheta_\Lambda}$ is a sum of squares matrix polynomial.

\item There exists a symmetry adapted basis of $H_{n,d} = \bigoplus_{\Lambda \vdash n}q_d^\Lambda \cdot S^{\Lambda}$, where the elements $\left(s_{1}^{\Lambda},\ldots,s_{q_d^\Lambda}^\Lambda\right)$ belonging to the isotypic component $q_d^\Lambda \cdot S^{\Lambda}$ are products each of one higher Specht polynomial and a monomial in $\psi_1,\ldots,\psi_n$. By defining matrix polynomials $B^\Lambda \in \left( \R[\underline{X}]^G\right)^{q_d^\Lambda\times q_d^\Lambda}$ via $B^\Lambda_{v,u} := \mathcal{R}_G(s_v^\Lambda \cdot s_u^\Lambda)$ a form $f \in H_{n,2d}^G$ is a sum of squares if and only if $$ f = \sum_\Lambda \Tr(B^{\Lambda} \cdot A_{\Lambda} )$$
for some positive semidefinite matrices $A_{\Lambda} \in \R^{q_d^\Lambda\times q_d^\Lambda}$.
\end{enumerate}
\end{thm}
\begin{proof}
 The isotypic decomposition of $H_{n,d}$ can be realized through multiplying the higher Specht polynomials of $G$ of degree $\leq d$ with products of fundamental invariants by theorems \ref{thm:irreducible reps} and \ref{thm:coinvariant algebra}. For every $k$ the multiplicity of $G$-modules isomorphic to $S^{\Lambda}$ in the subspace of the coinvariant algebra of degree $k$ is precisely $h_{k}^\Lambda$, while $N_G(d-k)$ gives the dimension of $H_{n,d-k}^{G}$. $(2)$ and $(3)$ follow now from Theorem \ref{thm:sums of squares description} and Corollary \ref{cor:sos}.
\end{proof}
\begin{remark}
For $D_n$ the isotypic decomposition in $(1)$ and the sizes of the matrices in $(2)$ and $(3)$ differ slightly, since then the $D_n$-module $S^{(\lambda,\lambda)}$ decomposes into two irreducible $D_n$-modules and $S^{(\lambda,\mu)}$ is $D_n$-isomorphic to $S^{(\mu,\lambda)}$. 
\end{remark}

We present how the isotypic decomposition of the $D_4$-module $H_{4,2}$ can be calculated using higher Specht polynomials:
\begin{example}
The $D_4$ fundamental invariants are the following:
\begin{align*}
    p_{2}=X_1^2+X_2^2+X_3^2+X_4^2,\;&\; p_{4} = X_1^4+X_2^4+X_3^4+X_4^4,\\
    p_6=X_1^6+X_2^6+X_3^6+X_4^6,\;&\; e_4=X_1X_2X_3X_4,
\end{align*}
i.e., we have $\R[\underline{X}]^{D_4} = \R[p_2,p_4,p_6,e_4]$. By Corollary \ref{cor:deomposition of polynomial ring} and Theorem \ref{thm:general sums of squares repr} the symmetry adapted basis for $H_{4,2}$ can be obtained by multiplication of the fundamental invariants with higher Specht polynomials (such that the degree equals $2$). \\
We apply Theorem \ref{thm:irreducible reps} to calculate the $D_4$ higher Specht polynomials. For a bipartition $\Lambda\vdash 4$ the minimal degree of a higher Specht polynomial associated with $\Lambda$ is given by the smallest number in $\{2\ch( T)+|\lambda^2|  : T \in \SYT (\Lambda)\}$. \\
Since the degrees of the fundamental invariants are at least $2$, we need to compute all higher Specht polynomials of degree $0$ and $2$. Therefore, we only need to consider partitions $(\lambda^1,\lambda^2) \vdash 4$ where $\lambda
^2 \vdash m, \, m \in \{ 0,2\}$ as otherwise the degree is odd. In the case $\lambda^2 \vdash 2$ it must be $\ch (T)=0$. This can only occur for $w(T)=(1,2,3,4)$. Which forces $\Lambda = ((2),(2))$. The possible remaining cases are $\Lambda^1 = ((4),\emptyset), \Lambda^2=((3,1),\emptyset),\Lambda^3=((2,2),\emptyset),\Lambda^4=((2,1,1),\emptyset),\Lambda^5=((1,1,1,1),\emptyset)$. We are looking for a standard (bi-)tableau $T$ of shape $\Lambda
^j, j \in \{1,2,3,4,5\}$ such that $\ch (T) \in \{0,1\}$. The case that it equals $0$ is only possible for $\Lambda^1$. In the remaining cases, $\ch (T)=1$ if and only if $T = \left( \ytableausetup{smalltableaux}
  \ytableaushort{
   123,4}\; , \; \emptyset
  \right)  $. Then the $D_4$-module $S^{((2),(2))}$ decomposes by Theorem \ref{thm:irreducible reps} into two irreducible, non-isomorphic modules $S_1^{((2),(2))}$ and $S_2^{((2),(2))}$. Hence, the $D_4$-module $H_{4,2}$ has the isotypic decomposition \begin{align*}
      H_{4,2} & =  S^{((4),\emptyset)} \oplus S^{((3,1),\emptyset)} \oplus S_1^{((2),(2))} \oplus S_2^{((2),(2))}.
  \end{align*}
  The relevant higher Specht polynomials are $1$ (for $S^{((4),\emptyset)}), X_4^2-X_1^2$ (for $S^{((3,1),\emptyset)}$) and $X_1X_2\pm X_3X_4$ (for $S_i^{((2),(2)}, i \in \{1,2\}$). 
\end{example}

\subsection{Stabilization of the isotypic decompositions}

In the following, we aim to prove a stabilization of the isotypic decompositions for the $Z_n$-modules $H_{n,d}$ for large $n$ and $(Z_n)_n \in \{(A_{n-1})_n,(B_n)_n,(D_n)_n\}$.

\begin{definition}
For a partition $\lambda= (\lambda_1,\lambda_2,\ldots,\lambda_l) \vdash n$ we write $\lambda +1:= (\lambda_1+1,\lambda_2,\ldots,\lambda_l) $ for the partition of $n+1$. For a bipartition $\Lambda \vdash n$ we write $\Lambda + 1:=(\lambda +1,\mu) \vdash n+1$. 
\end{definition}

We use the combinatorial description of the degrees of a symmetry adapted basis of $H_{n,d}$ from Remark \ref{rmk:deg of higher specht}. For $A_{n-1}$ and a standard tableau $T$ we have $\deg F_T^T = \ch (T)$, while for $B_n$ and $(T,S)$ it is $\deg \widehat{F}_{(T,S)}^{(T,S)} = 2 \ch (T,S) + |\mu|.$ Our aim is to identify the relevant standard (bi-)tableaux. 

\begin{lemma} \label{le:lemma1}
Let $\lambda \vdash n=d+k$ be a partition. In the case that the first row of a tableau $T \in \SYT (\lambda)$ does not begin with $1,2,\ldots,k$, then $\deg F_T^T > d$.
\end{lemma}
\begin{proof}
We assume that a standard tableau $T$ of shape $\lambda$ does not contain $1,2,\ldots,k$ in the first row. Let $\tilde{k}$ be the first entry of $T$ in the second row. It must be $\tilde{k}\leq k$ and $i(T)$ does contain at least $n-\tilde{k}+1$ entries which are larger than or equal to $1$. Therefore, $$\deg F_T^T = \ch  (T) \geq n-\tilde{k}+1\geq n-k+1 = d+1.$$ 
\end{proof}
We formulate Lemma \ref{le:lemma1} for bipartitions.
\begin{lemma}\label{le:lemma2}
Let $(\lambda,\mu) \vdash n$ be a bipartition, where $|\mu| \leq d$ and $|\lambda|\geq  \frac{d-1}{2}  +j$. Let $(T,S)$ be a standard bitableau of shape $(\lambda,\mu)$ where $\alpha_1< \ldots < \alpha_{|\lambda|}$ are all the entries in $T$. Assume that the first row of $T$ does not begin with $\alpha_1, \ldots ,\alpha_j$ then $\deg \widehat{F}^{(T,S)}_{(T,S)} > d.$
\end{lemma}
\begin{proof}
Assume that for some $i \leq j$ the $i$-th entry in the first row of $T$ is not $\alpha_i$ and let $i$ be minimal with this property. Then $\alpha_i$ must be the first entry in the second row and $|\lambda|-i+1$ entries in $i(T,S)$ are at least $1$. Hence $$ \deg \widehat{F}^{(T,S)}_{(T,S)} = 2 \ch (T,S) + |\mu| \geq 2(|\lambda|-i+1) \geq 2\left(  \frac{d-1}{2}  +j-j+1\right) \geq d+1.$$
\end{proof}

We write $T=(\alpha_{ij})$ for a standard tableau of shape $\lambda$, where $\alpha_{ij}$ denotes the entry in the i-th row and j-th coloumn of $T$, counted from the left to the right and the top to the bottom. Analogously, we write $(T,S)=\left( (\alpha_{ij}), (\beta_{ij})\right)$ for a standard tableau of a bipartition.

\begin{definition}
For a partition $\lambda=(\lambda_1,\ldots,\lambda_l) \vdash n=d+k$ with $\lambda_1 \geq k$, we define 
\begin{align*}
    \Pi_k^\lambda := \{ (\alpha_{ij}) \in \SYT (\lambda) : \alpha_{1j}=j, 1 \leq j \leq k \}.
\end{align*}
For a bipartition $\Lambda \vdash n=d+k$ we define
\begin{align*}
    \Pi_{k}^\Lambda := \{ (T,S) = ((\alpha_{ij}),(\beta_{ij})) \in \SYT ( \Lambda) : T_1\text{ starts with the $k$ smallest integers in } \{\alpha_{ij}\} \},
\end{align*}
where $T_1$ denotes the first row of $T$.
\end{definition}

\begin{example}
$$ \Pi_3^{((3,1),(1))} = \left\{  \left( \ytableausetup{smalltableaux}
  \ytableaushort{
   234,5}\; ,
  \ytableaushort{
   1} \right),
   \left( \ytableausetup{smalltableaux}
  \ytableaushort{
   134,5}\; ,
  \ytableaushort{
   2} \right), 
   \left( \ytableausetup{smalltableaux}
  \ytableaushort{
   124,5}\; ,
  \ytableaushort{
   3} \right),
     \left( \ytableausetup{smalltableaux}
  \ytableaushort{
   123,5}\; ,
  \ytableaushort{
   4} \right),
        \left( \ytableausetup{smalltableaux}
  \ytableaushort{
   123,4}\; ,
  \ytableaushort{
   5} \right)
   \right\}$$
\end{example}

\begin{lemma} \label{le:lemma3}
Let $n=d+k$, $\lambda = (\lambda_1,\ldots,\lambda_l) \vdash n$ be a partition and $$\abb{\rho_{n,n+1}^\lambda}{\Pi_k^\lambda}{\Pi_{k+1}^{\lambda +1}}{S=(\alpha_{ij})}{\widetilde{S}=(\widetilde{\alpha}_{ij})},$$ where $\widetilde{\alpha}_{1i}=i$ for $1 \leq i \leq \alpha_{21}$. Further, $\widetilde{\alpha}_{ji}=\alpha_{ji-1}+1$ for $j=1,i\geq \alpha_{21}+1$ and all $(j,i)$ with $j \geq 2$. \\
Then $ \rho_{n,n+1}^\lambda$ is injective and $i(S),i(\widetilde{S})$ differ only by a zero, i.e., any non-zero entry in $i(S)$ occurs with the same multiplicity in $i(\widetilde{S})$, while $0$ occurs one more time. Furthermore, if $k > d-1$ then any $\widetilde{S} \in \Pi_{k+1}^{\lambda+1} \setminus \rho_{n,n+1}^\lambda (\Pi_k^\lambda)$ has $\ch (\widetilde{S})> d$.
\end{lemma}
\begin{proof}
Since $S \in \Pi_k^\lambda$ is standard, we observe that $\alpha_{21}$ is the smallest integer $i$ for which $\alpha_{1i} \neq i$. For ${S} \in \Pi_k^\lambda$ the tableau $\widetilde{S}$ of shape $\lambda +1$ is indeed standard:  $\widetilde{S}$ is filled with $1,\ldots,n+1$. Increasing rows and columns are inherited from $S$, as $\alpha_{1 \alpha_{21}} > \alpha_{21}$. $\widetilde{S}$ is clearly increasing in any column from the second row onward. But also from the first row to the second. For $1 \leq i \leq \alpha_{21}$ this is clear from $S$. For $i > \alpha_{21}$ this follows because $\widetilde{\alpha}_{1i} = \alpha_{1,i-1} +1 < \alpha_{2,i-1} +1 < \alpha_{2,i}+1 = \widetilde{\alpha}_{2i}$. \\
The smallest $j$ which is written left of $j-1$ in $w(S)$ (resp. $w(\widetilde{S})$) is $\alpha_{21}$ (resp. $\widetilde{\alpha}_{21}=\alpha_{21}+1$). From there any $j> \alpha_{21}$ is left of $j-1$ in $w(S)$ if and only if $j+1$ is left of $j$ in $w(\widetilde{S})$. Hence, $i(S)$ and $i(\widetilde{S})$ differ only by a zero. \\
Consider $\psi_{n+1,n}^{\lambda+1} : \Pi_{k+1}^{\lambda +1} \rightarrow \YT (\lambda)$ which maps a standard tableau $\widetilde{S}$ to a tableau $S$  by removing the box of the first entry $\widetilde{\alpha}_{1j}$ in the first row of $\widetilde{S}$, that is strictly smaller than $\widetilde{\alpha}_{1j+1}-1$ (otherwise the last entry). The entries to the right are shifted to the left. Any entry that was to the right of $\widetilde{\alpha}_{1j}$ or in a lower row is decreased by one. If $\psi_{n+1,n}^{\lambda+1}(\widetilde{S})=:S$ is again standard, then $\psi_{n+1,n}^{\lambda+1} \circ \rho_{n,n+1}^\lambda(S) =S$. This shows the injectivity of $\rho_{n,n+1}^\lambda$. \\
If $S$ is not standard, then one entry in the first column must be smaller than the entry below. Assume that this happens at $S$'s entry $\alpha_{1j}$. By assumption $j > k$, but this means $\lambda_2 \geq j > k$. Notice that $$\ch (\widetilde{S}) \geq \lambda_2 + 1 \geq k+2 \geq d+1.$$ \\

\end{proof}

\begin{lemma} \label{le:lemma4}
Let $n=d+k$, $\Lambda = (\lambda,\mu)\vdash n$ a bipartition, where $\lambda=(\lambda_1,\ldots,\lambda_l),$ and $$ \abb{\rho_{n,n+1}^\Lambda}{\Pi^\Lambda_{k}}{\Pi^{\Lambda+1}_{k+1}}{(T,S)=((\alpha_{ij}),(\beta_{ij}))}{(\widetilde{T},\widetilde{S})=((\widetilde{\alpha_{ij}}),(\widetilde{\beta}_{ij}))},$$ where $(\widetilde{T},\widetilde{S})$ is defined by: Let $i$ be minimal with $\alpha_{1i} \neq i$, then $\widetilde{\alpha}_{1j}=j,$ $1 \leq j \leq i$ and $\widetilde{\alpha}_{1k}=\alpha_{1k-1}+1,$ for $ i+1 \leq k \leq \lambda_1+1$, $\widetilde{\alpha}_{jk} = \alpha_{jk}+1,$ when $j \geq 2$, and $\widetilde{\beta}_{jk}=\beta_{jk}+1$. If such an $i$ does not exist, then $\widetilde{\alpha}_{1k}=k$, $\widetilde{\alpha}_{ji} =\alpha_{ji}+1, j \geq 2$ and $\widetilde{\beta}_{ji}= \beta_{ji}+1$. \\
Then $\rho_{n,n+1}^{\Lambda}$ is injective and $i(S,T), i(\widetilde{S},\widetilde{T})$ differ only by a zero, i.e., any non-zero entry in $i(S,T)$ occurs with the same multiplicity in $i(\widetilde{S},\widetilde{T})$ and $0$ occurs one more time. Furthermore, if $k >  \frac{d}{2} -2$ then any $(\widetilde{T},\widetilde{S}) \in \Pi^{\Lambda+1}_{k+1} \setminus \rho_{n,n+1}^\Lambda (\Pi^\Lambda_{k})$ has $2\ch (\widetilde{T},\widetilde{S}) > d$.
\end{lemma}
\begin{proof}
For $(T,S) \in \Pi_{k}^\Lambda$ $(\widetilde{T},\widetilde{S})$ is indeed a standard bitableau of shape $\Lambda +1$, since each increasing entry in every row and column is inherited from $(T,S)$. An integer $j$ occurs left of $j-1$ in $w(T,S)$ if and only if $j+1$ occurs left of $j$ in $w(\widetilde{T},\widetilde{S})$. In particular, $i(T,S)$ and $i(\widetilde{T},\widetilde{S})$ differ only by an additional zero entry and hence their charges are equal.
Consider $f : \Pi_{k+1}^{\Lambda +1} \rightarrow \text{YT}(\Lambda)$ which maps an element $(\widetilde{T},\widetilde{S}) \in \Pi_{k+1}^{\Lambda +1}$ to a tableau of shape $\Lambda$ by removing $\widetilde{\alpha}_{11}$, if $\widetilde{\alpha}_{11} \neq 1$ and otherwise, the box containing the largest entry in the first row of $\widetilde{T}$ that is not the predecessor of the following number, and subtracting $1$ from any larger entry $\widetilde{\alpha}_{ji},\widetilde{\beta}_{ji}$. Then $f$ is the inverse of $\rho_{n,n+1}^\Lambda$ and therefore $\rho_{n,n+1}^\Lambda$ is injective. \\
If $f(\widetilde{T},\widetilde{S})$ is not standard, then $\lambda_2 \geq k+1$. For $k > \frac{d}{2} -2$ it is $$2\ch (\widetilde{T},\widetilde{S}) \geq 2(k+2) > d.$$ 
\end{proof}

\begin{definition}
For $m>n\geq d$ and partitions $\Lambda,\lambda \vdash n$ we write $\rho_{n,m}^\lambda := \rho_{m-1,m}^{\lambda + m-n-1} \circ \cdots \circ \rho_{n,n+1}^{\lambda}$ and $\rho_{n,m}^\Lambda := \rho_{m-1,m}^{\Lambda + m-n-1} \circ \cdots \circ \rho_{n,n+1}^{\Lambda}$.
\end{definition}

We now are in the position to prove the following stabilization result, which was already proven in \cite{phdthesis,riener2013exploiting} for the case of  the symmetric group.

\begin{thm} \label{thm:ModulesStabilize}
Let $n \in \N, \Lambda \vdash n$ and $Z_n \in \{A_{n-1},B_n,D_n\}$. For large enough $n$ the $Z_n$- and $Z_{n+1}$-isotypic decompositions remain stable. In the sense that $S^{(\lambda,\mu)}$ occurs with the same multiplicity in $H_{n,d}$ as $S^{(\lambda +1,\mu)}$ in $H_{n+1,d}$, where $\lambda +1 := (\lambda_1+1,\lambda_2,\ldots,\lambda_l)$. The stabilization of the isotypic decomposition of $H_{n,d}$ occurs at least from $n = 2d$ for $A_{n-1}$, $n = d$ for $B_n$ and $n > 2d$ in the case of $D_n$. 
\end{thm}

\begin{proof}
We restrict us to the cases $A_{n-1}$ and $n \geq 2d$, and $B_n$ where $n \geq d$. For $n > 2d$ the relevant fundamental invariants of degree $\leq d$ are equal for $B_n$ and $D_n$. The same argument as in the $B_n$ case applies, since no bipartition can be of the form $(\lambda,\lambda)$. By iteration, it is sufficient to compare the isotypic decompositions of $H_{n,d}$ and $H_{n+1,d}$. 

Let $n \geq 2d$ and $\Lambda=(\lambda,\mu) \vdash n$ be a bipartition with $ |\mu|\leq d$ (resp. $\lambda \vdash n$ a partition in the case of $A_{n-1}$). Further, be $f_1,\ldots,f_m$ a symmetry adapted basis for the higher Specht polynomials of the $Z_n$-module $\bigoplus_{i=1}^m S^\Lambda$ (resp. $\bigoplus_{i=1}^m S^\lambda$) from Theorem \ref{thm:irreducible reps}. I.e., there exist $m$ many standard (bi-)tableaux $T:=T_1,T_2\ldots,T_m$ of shape $\Lambda$ (resp. $\lambda$) and $f_j = \pi \widehat{F}_{T}^{T_j}$ (resp. $f_j = \pi F_T^{T_j}$), for some $\pi \in \R[\underline{X}]^{Z_n}$. $\pi$ can be chosen as a product of fundamental invariants of $Z_n$ by a change of basis, since $\pi f_j$ must be homogeneous. The degree of a polynomial $f_j$ is determined by $d_1,\ldots,d_n$, the charge of a standard (bi-)tableau $T_j$ and $|\mu|$.

The relevant degrees of fundamental invariants are equal for $n$ and $n+1$. By Lemma \ref{le:lemma1} $T_1,\ldots,T_m \in \Pi_{n-d}^\lambda$ and by Lemma \ref{le:lemma3} for any $i$  $\rho_{n,n+1}^\lambda(T_i)$ is a standard tableau with same charge. Furthermore, the map $\rho_{n,n+1}^\lambda$ is injective and any standard tableau that is not contained in the image has too large charge. The claim follows, as only standard tableau in $\rho_{k,k+1}^{\lambda}( \Pi_{k}^\lambda)$ are possible options for higher Specht polynomials in $H_{n+1,d}$. \\
By the Lemmas \ref{le:lemma2} and \ref{le:lemma4} the standard bitableaux $(T,S)$ of shape $\Lambda$ with $2 \ch (T,S) \leq d $ are in bijection with the standard bitableaux $(\widetilde{T},\widetilde{S})$ of shape $\Lambda +1$ with $ 2 \ch (\widetilde{T},\widetilde{S}) \leq d$ and the bijection preserves the charge. Furthermore, our bijection adds a zero to the index of the image tableaux and preserves the other entries. This proves already the claim.
\end{proof}

We note that in the case of $D_n$ and $n=d$, an additional fundamental invariant of degree $d$ occurs, which does not occur for $n > d$ anymore. Thus, at least the trivial representation occurs with larger multiplicity in $H_{d,d}$ than in $H_{d+1,d}$. However, Example \ref{ex: d->d+1} shows that already for the symmetric group the stabilization does not occur in the step from $d$ to $d+1$ in general. 
\begin{example} \label{ex: d->d+1}
Consider the bitableau $T= \ytableausetup{smalltableaux}
  \ytableaushort{
   125,34}\; $
  of shape $\lambda+1 = (3,2) \vdash 5$. It is $\ch (T) = 3$, i.e., $p_1 F^T_T \in H_{5,4}$. However, $\text{SYT}(\lambda) = \left\{  \ytableaushort{
   12,34}\; 
  ,  \ytableaushort{
   13,24}\;
  \right\}$ with charges $2$ and $4$. For any $S \in \text{SYT}(\lambda)$ we can construct a tableau $\widetilde{S} \in \text{SYT}(\lambda+1)$ with the same charge, but $T$ cannot be obtained in this way. In particular, the $A_3$-module $S^{\lambda}$ has smaller multiplicity in $H_{4,4}$ than the $A_4$-module $S^{\lambda +1}$ in $H_{5,4}$.
\end{example}

\begin{cor} \label{cor:ModulesStabilize}
For a fixed degree $d \in \N$ and a sequence $(Z_n)_n$ of reflection groups 
$(A_{n-1})_n$ or $ (B_n)_n$ the sums of squares decomposition in $H_{n,2d}^{Z_n}$ for $n \geq 2d$ (for $A_{n-1}$) and $n \geq d $ (for $B_n$) are equal up to the map $\rho_{n,m}^\Lambda$, i.e., up to $\rho_{n,m}^\Lambda$ the same matrix polynomials can be used in a sum of squares representation. The same stabilization and equality occurs for the sequence $(D_n)_n$ when $n > 2d$.
\end{cor}
\begin{proof}
This follows from Theorem \ref{thm:ModulesStabilize} and Lemmas \ref{le:lemma1}, \ref{le:lemma2}, \ref{le:lemma3}, \ref{le:lemma4}.
\end{proof}

The case $n=2d$ is the last, where $\Lambda \vdash n$ can be of the form $(\lambda, \lambda)$, i.e., the $D_{2d}$-module $S^{\Lambda}$ is not irreducible in $H_{n,d}$ but the $D_{2d+1}$-module $S^{\Lambda+1}$ is irreducible in $H_{n+1,d}$ (see Theorem \ref{thm:irreducible reps}). Nevertheless, the multiplicities in $H_{2d,d}$ and $H_{n,d}$ are equal for $n \geq 2d$. Moreover, whenever $n \geq d$ for $B_n$, or $n > d$ in case of $D_n$ one can use that if $S^{\Lambda}\subset H_{n,d}$, for $\Lambda = (\lambda, \mu) \vdash n$, and $d$ even (odd), then $|\mu|$ must also be even (odd).

\section{Concrete examples and applications}
In this section, we apply the presented techniques from the preceding section $3$ to solve non-negativity versus sums of squares questions. In contrast to the non-equivariant case, the $B_n$-invariant forms have a non-trivial equality of the sets of even symmetric sums of squares and non-negatives in $3$ variables and degree $8$. This was proven by Harris \cite{harris1999real}. In fact, it turns out that this is the only non-trivial equality case \cite{goel2017analogue}. We will present a characterization of the dual and primal cones of $B_3$-invariant sum of squares ternary octics and obtain a new elementary proof of Harris' theorem. Moreover, we study $D_n$-invariant forms, prove that $\mathcal{P}_{4,4}^{D_4}$ is a simplicial cone and answer the non-negativity versus sums of squares question there.

In general, testing non-negativity of a polynomial in more than two variables is already for quartics an NP-hard problem (see e.g., \cite{blum1998complexity} or \cite{murty1985some}). In equivariant situations, it is therefore of interest to exploit the symmetry of invariant polynomials to reduce this complexity. The works in \cite{acevedo2016test,friedl2018reflection,harris1999real,moustrou2019symmetric,riener2012degree,riener2016symmetric,timofte2003positivity} focus on providing \emph{test sets} for verification of non-negativity of invariant polynomials. In particular, it is known  that for reflection groups the value set of invariant polynomial functions of certain degrees the set of values can be  determined by evaluation on subspaces of the hyperplane arrangement (see \cite{acevedo2016test,friedl2018reflection} for details). 

We remark that each element in the infinite series $I_2(m)$ of dihedral groups, has only one action on $\R^2$. In particular is any $I_2(m)$ invariant non-negative form a sum of squares.

\subsection{Even symmetric octics}\label{sec:even}
One of the well known and rare cases of equality of sums of squares and non-negative forms in equivariant situations was proven by Harris in \cite{harris1999real}. Harris' proof is quite analytical. In this subsection we derive a lower dimensional test set for non-negativity of even symmetric ternary octics and as a byproduct we give a new proof of equality. Furthermore, we present a uniform description of the cones of $n$-ary even symmetric sums of squares octics.

 \begin{thm} \label{Thm:DualEvenSymOct} The dual cone of even symmetric ternary octics sums of squares has the following description
 \begin{align*}
     \left( \Sigma_{3,8}^{B_3} \right)^\ast
     & = \left\{\ev_{\left(a,\sqrt{1-a^2},0\right)},\ev_{(b,c,c)} : \frac{1}{2} \leq a \leq 1, 0 \leq b \leq 1, c = \sqrt\frac{{(1-b^2)}}{{2}}\right\}.
 \end{align*}
  \end{thm}

As a consequence of Theorem \ref{Thm:DualEvenSymOct} we can give a new proof for Harris' result on even symmetric ternary octics.
\begin{cor}\cite[Theorem 4.1]{harris1999real}\label{COR Harris} The sets of non-negative even symmetric ternary octics and sums of squares are equal, i.e.,
$\Sigma_{3,8}^{B_3} = \mathcal{P}_{3,8}^{B_3}$.
\end{cor}
\begin{proof}
By Theorem \ref{Thm:DualEvenSymOct} the cone $\left( \Sigma_{n,2d}^G\right)^\ast$ is generated by point-evaluations. The claim follows from Corollary \ref{cor:Sigma=P}.
\end{proof}
In the following, we provide a study of the even symmetric sums of squares ternary octics.

\begin{lemma}\label{lemma:Harris}
The $B_3$-module $H_{3,4}$ has the isotpyic decomposition \[ H_{3,4} = 2 \cdot {S}^{((3),\emptyset)} \oplus 2 \cdot {S}^{((2,1),\emptyset)} \oplus 2 \cdot {S}^{((1),(2))} \oplus \mathcal{S}^{((1),(1,1))}. \]
A symmetry adapted basis for $H_{3,4}$ realising the $B_3$-isotypic decomposition is given by the following polynomials:
\begin{align*}
    S^{((3),\emptyset)} & :  \left\{e_1(\underline{X}^2)^2,e_2(\underline{X}^2)\right\}, &
    S^{((2,1),\emptyset)} & :  \left\{e_1(\underline{X}^2)(X_3^2-X_1^2), X_2^2X_3^2-X_1^2X_2^2 \right\}, \\
    S^{((1),(2))} & : \left\{e_1(\underline{X}^2)X_2X_3, X_1^2X_2X_3\right\}, &
    S^{((1),(1,1))} & :  \left\{(X_3^2-X_2^2)X_2X_3\right\}.
\end{align*}
\end{lemma}
\begin{proof}
We need to determine the multiplicity of the irreducible $B_3$-modules $S^{(\lambda,\mu)}$ in $H_{3,4}$ for any bipartition $(\lambda,\mu) \vdash 3$. We can immediately exclude some bipartitions: Since we need only higher Specht polynomials of degree $0,2$ or $4$ by Theorem \ref{thm:general sums of squares repr}, the degree - which equals $2$ times the charge of a standard bitableau of shape $(\lambda,\mu)$ plus $|\mu|$ - must be $0,2$ or $4$. However, this implies that only bipartitions with $\mu \in \{ \emptyset, (2),(1,1)\}$ are feasible to obtain an even degree. By going through all the remaining cases one obtains precisely the following higher Specht polynomials of degree $0,2$ and $4$:
$$\left\{1, X_3^2-X_1^2, X_2^2X_3^2-X_1^2X_2^2, 
X_2X_3,X_1^2X_2X_3,(X_3^2-X_2^2)X_2X_3.\right\}$$
Multiplying by the invariants $1, e_1(\underline{X}^2)^2$ and $e_2(\underline{X}^2)$ results accordingly in the above-mentioned symmetry adapted basis.
\end{proof}

\begin{cor}\label{cor:ternaryOcticsSOSDec}
An even symmetric ternary octic $f \in H_{3,8}^{B_3}$ is a sum of squares if and only if there exist positive semidefinite matrices $A^{(1)},A^{(2)},A^{(3)} \in \R^{2 \times 2}$ and $A^{(4)} \in \R^{1 \times 1}$ such that $$f = \langle A^{(1)} B^{(1)} \rangle +  \langle A^{(2)} B^{(2)} \rangle + \langle A^{(3)} B^{(3)} \rangle + \langle A^{(4)} B^{(4)} \rangle,$$ where $B^{(j)}$ are the following matrix polynomials corresponding to the $B_3$-modules in $H_{3,4}$ \begin{align*}
    B^{(1)} & := \left( \begin{array}{cc}
e_1(\underline{X}^2)^4     & e_1(\underline{X}^2)^2e_2(\underline{X}^2) \\
e_1(\underline{X}^2)^2e_2(\underline{X}^2)     & e_2(\underline{X}^2)^2
\end{array} \right), \\ 
B^{(2)} & := \left( \begin{array}{cc}
\frac{2}{3}e_1(\underline{X}^2)^4-2e_1(\underline{X}^2)^2e_2(\underline{X}^2)     &  -3e_1(\underline{X}^2)e_3(\underline{X}^2)+\frac{1}{3}e_1(\underline{X}^2)^2e_2(\underline{X}^2)\\
  -3e_1(\underline{X}^2)e_3(\underline{X}^2)+\frac{1}{3}e_1(\underline{X}^2)^2e_2(\underline{X}^2)  & \frac{2}{3}e_2(\underline{X}^2)^2-2e_1(\underline{X}^2)e_3(\underline{X}^2)
\end{array} \right),  \\
B^{(3)} & := \left( \begin{array}{cc}
\frac{1}{3}e_1(\underline{X}^2)^2e_2(\underline{X}^2)     & e_1(\underline{X}^2)e_3(\underline{X}^2)  \\
e_1(\underline{X}^2)e_3(\underline{X}^2)     &  \frac{1}{3}e_1(\underline{X}^2)e_3(\underline{X}^2)
\end{array}\right),  \\ 
B^{(4)} & := \left( e_1(\underline{X}^2)e_3(\underline{X}^2)-\frac{4}{3}e_2(\underline{X}^2)^2+\frac{1}{3}e_1(\underline{X}^2)^2e_2(\underline{X}^2) \right).
\end{align*}
\end{cor}
\begin{proof}
The matrices $B^{(1)},\ldots,B^{(4)}$ are the symmetrizations of the products of the symmetry adapted basis from Lemma \ref{lemma:Harris}. By Theorem \ref{THM Decomp} any invariant sum of squares form has such a representation.
\end{proof}

\begin{cor}\label{cor:ternaryOcticsDualToSOSDecomp}
A linear form $\ell \in \left( H_{3,8}^{B_3} \right)^\ast$ is contained in $\left( \Sigma_{3,8}^{B_3} \right)^\ast$ if and only if the following matrices are positive semidefinite
\begin{align*}
    \left( \begin{smallmatrix}
      m_{(1^4)}   & m_{(2,1^2)} \\
    m_{(2,1^2)}     & m_{(2^2)}
   \end{smallmatrix}\right), 
    \left( \begin{smallmatrix}
      \frac{2}{3}m_{(1^4)}-2m_{(2,1^2)}  & \frac{1}{3}m_{(2,1^2)}-3m_{(3,1)} \\
    \frac{1}{3}m_{(2,1^2)}-3m_{(3,1)}     & \frac{2}{3}m_{(2^2)}-2m_{(3,1)}
    \end{smallmatrix}\right), 
    \left( \begin{smallmatrix}
    \frac{1}{3}m_{(2,1^2)}     & m_{(3,1)} \\
    m_{(3,1)}     & \frac{1}{3}m_{(3,1)}
    \end{smallmatrix}\right), 
    \left( \begin{smallmatrix}\frac{1}{3}m_{(2,1^2)}-\frac{4}{3}m_{(2^2)}+m_{(3,1)}\end{smallmatrix} \right),
\end{align*}
where we write $m_{(1^4)} := \ell (e_1(\underline{X}^2)^4), m_{(3,1)} := \ell (e_1(\underline{X}^2)e_3(\underline{X}^2)), m_{(2,1^2)} := \ell (e_1(\underline{X}^2)^2e_2(\underline{X}^2))$ and $m_{(2^2)} := \ell (e_2(\underline{X}^2)^2)$.
\end{cor}
\begin{proof}
This is precisely the dual statement to Corollary \ref{cor:ternaryOcticsSOSDec} by Lemma \ref{le:psd}.
\end{proof}

 \begin{remark}\label{rmk:isotypic decomp H38Bn}
 We observe that \[ H_{3,8}^{B_3} = \langle p_2^4, p_2^2 p_4, p_2p_6, p_4^2 \rangle_\R = \langle e_1(\underline{X}^2)^4, e_1(\underline{X}^2)e_3(\underline{X}^2), e_1(\underline{X}^2)^2e_2(\underline{X}^2), e_2(\underline{X}^2)^2 \rangle_\R \] is a $4$-dimensional $\R$-vector space. We pick as the fundamental invariants the elementary symmetric polynomials evaluated in $\underline{X}^2=(X_1^2,X_2^2,X_3^2)$ and work with the $\R$-basis $$\left(e_1(\underline{X}^2)^4, e_1(\underline{X}^2)e_3(\underline{X}^2), e_1(\underline{X}^2)^2 e_2(\underline{X}^2), e_2(\underline{X}^2)^2\right)$$ of $H_{3,8}
^{B_3}$. We study explicitly the extremal elements in $\left( \Sigma_{3,8}^{B_3}\right)^\ast$ and show that all of them are point-evaluations which is then used to prove Theorem \ref{Thm:DualEvenSymOct}. In the remaining part of this subsection we will always use the following notation for an extremal element $\ell \in \left( \Sigma_{3,8}^{B_3}\right)^\ast$. $\mathcal{Q}_\ell$ denotes the associated $B_3$-invariant quadratic form on $H_{3,4}$, $W_\ell := \ker \mathcal{Q}_\ell$ its kernel and $$W^{\langle 2 \rangle}_\ell := \ker \ell = \left\{ h\in H_{n,2d}^G\,:\;h=\mathcal{R}_G \left(\sum f_i g_i\right)\;\text{ with } f_i \in W \text{ and  }g_i \in H_{n,d}\right\}$$ (see Proposition \ref{prop:W2}). A hyperplane in $H_{3,8}^{B_3}$ is of dimension $3$, hence from Lemma \ref{le:MaxOfKernel} we know that $\dim W^{\langle 2 \rangle}_\ell = 3$. By Lemma \ref{lemma:Harris} the isotypic decomposition of the $B_3$-submodule $W_\ell$ of $H_{3,4}$ has the form \begin{align*}
    W_\ell = \ker \mathcal{Q}_\ell = \alpha \cdot {S}^{((3),\emptyset)} \oplus \beta \cdot {S}^{((2,1),\emptyset)} \oplus \gamma \cdot {S}^{((1),(2))} \oplus \delta \cdot {S}^{((1),(1,1))},
\end{align*} where $\alpha, \beta, \gamma \in \{0,1,2\}$ and $\delta \in \{0,1\}$. \\
We make frequently use of the fact that $\ker \ell$ is maximal among any kernel of elements in $\left( \Sigma_{3,8}^G\right)^\ast$, i.e., when $\ker \ell$ contains a non trivial zero then $\ell$ must be a scalar of the point-evaluation at this point (see Lemma \ref{le:MaxOfKernel}).
 \end{remark}
 In the following lemmas we do case distinctions on $\alpha,\beta,\gamma$ and $\delta$ to obtain a classification of all extremal elements in the dual cone $\left( \Sigma_{3,8}^{B_3} \right)^\ast$.

 \begin{lemma}\label{le:38Teil1}
 Let $\ell \in \left( \Sigma_{3,8}^{B_3} \right)^\ast$ be an extremal element. Then $\alpha < 2$, i.e., the multipilcity of the trivial representation in $W_\ell$ is smaller than 2.
 \end{lemma}
 \begin{proof}
 If $\alpha = 2$ then $e_1(\underline{X}^2)^2 \in W_\ell$ and hence $e_1(\underline{X}^2)^4 \in  W_\ell^{\langle 2 \rangle} = \ker \ell$. However, any monomial of degree $8$ that is a square occurs with positive coefficients in $e_1(\underline{X}^2)^4$, which implies $\ell = 0$ must be the zero map. 
 \end{proof}
 
 \begin{lemma}\label{le:38Teil2}
 Let $\ell \in \left( \Sigma_{3,8}^{B_3} \right)^\ast$ be an extremal element and $\alpha = 0$. Then $\ell$ is a scalar of the point-evaluation $\ev_z$, where $z \in \{(1,1,1),(1,0,0),(1,1,0)\}.$
 \end{lemma}
 \begin{proof}
 In the case $\beta = 2$ we know by dimension reasons on $W_\ell^{\langle 2 \rangle}$ that any other $B_3$-module occurring in $W_\ell$ must already be contained in $2 \cdot S^{((2,1),\emptyset)}$. However, the forms in the module $2 \cdot S^{((2,1),\emptyset)}$ have the common zero $(1,1,1)$. \\
 If $\beta = 1$, then it must be $\gamma \geq 1$ or $\delta = 1$ such that $W_\ell^{\langle 2 \rangle}$ is a hyperplane. For $\delta = 1$ the elements in $W_\ell$ have the common root $(1,1,1)$. Now, we consider the case $\beta = 1,\gamma \geq 1$. Thus for some pairs $(a,b),(c,d) \in \R^2\setminus\{(0,0)\}$ $$ae_1(\underline{X}^2)(X_3^2-X_1^2)+b(X_2^2X_3^2-X_1^2X_2^2),ce_1(\underline{X}^2)X_2X_3+d X_1^2X_2X_3 \in W_\ell,$$ and their symmetrized products with elements in $H_{3,4}$ are contained in $W_\ell^{\langle 2 \rangle}$, i.e., 
 \begin{align*}
 0=& a\left(\frac{2}{3}m_{(1^4)}-2m_{(2,1^2)}\right) + b \left( \frac{1}{3}m_{(2,1^2)}-3m_{(3,1)} \right),  \\ 
 0=& a\left( \frac{1}{3}m_{(2,1^2)}-3m_{(3,1)}\right) + b \left( \frac{2}{3}m_{(2^2)}-2m_{(3,1)} \right), \\
 0=&\frac{c}{3}m_{(2,1^2)}+dm_{(3,1)}, \\
 0=& cm_{(3,1)}+\frac{d}{3}m_{(3,1)}. 
 \end{align*}
  We now distinguish between $m_{(3,1)}$ equals or not equals zero:
 \begin{itemize}
     \item[i)] In the case that $m_{(3,1)} \neq 0$ we have that $c + \frac{d}{3}=0$. Since $W_\ell$ is a linear space we can set $c=1$ and $d=-3$. However, then the $B_3$-module $W_\ell$ has the common zero $(1,1,1)$. Thus $\ell$ is a scalar of the point-evaluation $\ev_{(1,1,1)}.$ 
     \item[ii)] Let $m_{(3,1)} =0$. We first assume that $c \neq 0$. Then $m_{(2,1^2)} = 0$ and since $m_{(1^4)} >0$ it is $a=0.$ Hence, $b \neq 0$ and $m_{(2^2)}=0$ which implies that the elements in $W_\ell$ all vanish at $(1,0,0)$ and $\ell$ is a scalar of $\ev_{(1,0,0)}$. \\
     If $c = 0$ we have \begin{align*}
    0=&  a\left(\frac{2}{3}m_{(1^4)}-2m_{(2,1^2)}\right) + b \left( \frac{1}{3}m_{(2,1^2)} \right), \\ 
    0=& a\left( \frac{1}{3}m_{(2,1^2)}\right) + b \left( \frac{2}{3}m_{(2^2)} \right). 
 \end{align*} If $a = 0$ then $\ell$ is a scalar of $\ev_{(1,0,0)}$, since any form in $W_\ell^{\langle 2 \rangle}$ has the zero $(1,0,0)$. Otherwise, we may assume that $a=1$ since $W_\ell^{\langle 2 \rangle}$ is a linear space. It is 
 \begin{align*}
     0=& \frac{2}{3}m_{(1^4)}+(-2+\frac{b}{3})m_{(2,1^2)}, \\  0=&\frac{1}{3}m_{(2,1^2)}+\frac{2b}{3}m_{(2^2)}.
 \end{align*} Through scaling of $\ell$ and $m_{(1^4)}>0$, we can assume that $m_{(1^4)} = 1$. If $b = 0$, then $0 =m_{(1^4)} = 1$ which cannot be true. So $b \neq 0$ and $m_{(2,1^2)} = \frac{2}{6-b}, m_{(2^2)} = \frac{1}{-6b+b^2}$, for a non zero $b \neq 6.$ From the positive semidefiniteness conditions in Corollary \ref{cor:ternaryOcticsDualToSOSDecomp} we obtain from the first matrix $$\det \left( \begin{array}{cc}
1     &m_{(2,1^2)}  \\
m_{(2,1^2)}     & m_{(2^2)}
\end{array} \right) \geq 0,$$ which implies that $ -2 \leq b < 0$. And the positive semidefiniteness of the last matrix in \ref{cor:ternaryOcticsDualToSOSDecomp} $$\frac{1}{3}m_{(2,1^2)}-\frac{4}{3}m_{(2^2)}+m_{(3,1)} \geq 0$$ implies that $b \leq -2$ or $0 < b < 6$. Thus $b = -2$ and $\ell$ is the point-evaluation $\ev_{(\frac{1}{\sqrt{2}},\frac{1}{\sqrt{2}},0)}.$
 \end{itemize}
 Finally, if $\gamma \geq 1$, then $\beta=1$ or $\delta=1$. However, we have already examinied the case $\beta=1$. For $\delta = 1$ the elements in $W_\ell$ have the common zero $(1,0,0)$. Thus $\ell$ is a scalar of  $\ev_{(1,0,0)}$.
 \end{proof}

 Therefore we proceed with the cases where ${\alpha = 1}$, which implies that $ae_1(\underline{X}^2)^2 + e_2(\underline{X}^2) \in W_\ell$ for an $a \in \R$, since $e_1(\underline{X}^2)^4 \not \in W_\ell$. This means for $\ker \ell$ \begin{align*}
  & a m_{(1^4)} + m_{(2,1^2)} = 0, \\ & am_{(2,1^2)}+m_{(2^2)} = 0.  
 \end{align*} Moreover, since $m_{(1^4)}> 0$ and $\ell$ is a linear form we can set without loss of generality $m_{(1^4)}=1$, as $\ell$ is then just a positive scalar. 
 The positive semidefinitness conditions with the reductions $m_{(2,1^2)} = -am_{(1^4)}, m_{(2^2)} = a^2m_{(1^4)}$ and $m_{(1^4)} =1$ become to
 \begin{align}\label{simpledmatrices}
     \left( \begin{smallmatrix}
      1    & -a \\
    -a      & a^2
     \end{smallmatrix} \right), 
     \left(\begin{smallmatrix} \frac{2}{3}+2a & -\frac{1}{3}a-3m_{(3,1)} \\ -\frac{1}{3}a-3m_{(3,1)} & \frac{2}{3}a^2-2m_{(3,1)}  \end{smallmatrix} \right), 
     \left( \begin{smallmatrix}
    -\frac{1}{3}a      &m_{(3,1)}  \\
    m_{(3,1)}      & \frac{1}{3}m_{(3,1)}
     \end{smallmatrix} \right), 
     \left( \begin{smallmatrix} \frac{-a}{3}-\frac{4a^2}{3}+m_{(3,1)}  \end{smallmatrix} \right) \succeq 0.
 \end{align}
 From the positive semidefiniteness of the second matrix and $-a=m_{(2,1^2)} \geq 0$ we obtain $a \in [\frac{-1}{3},0]$.

We now proceed with a case distinction on the paramaters $\beta,\gamma,\delta$:

 \begin{lemma}\label{le:38Teil3}
Let $\ell \in \left(\Sigma_{3,8}^{B_3}\right)^\ast$ be an extremal element. If $\alpha = \delta = 1$, then $\ell$ is a scalar of a point-evaluation in $(1,1,0)$.
 \end{lemma}
 \begin{proof}
$\delta=1$ means that $S^{((1),(1,1))} \subset W_\ell$ which implies $(X_3^2-X_2^2)X_2X_3 \in W_\ell$ and  $$-\frac{a}{3}-\frac{4a^2}{3}+m_{(3,1)}=0.$$ Positiveness yields $0 \leq m_{(3,1)} = \frac{1}{3}(a+4a^2)$ and therefore that $ a \leq -\frac{1}{4}$.
We use that the determinant of the second matrix in (\ref{simpledmatrices}) is non-negative, i.e., $$ 0 \leq \left(\frac{2}{3}+2a\right)\left(\frac{2}{3}a^2-2m_{(3,1)}\right)-\left( -\frac{1}{3}a-3m_{(3,1)} \right)^2 = -\frac{4}{9} a (1 + 3 a)^2 (1 + 4 a).$$ This is not satisfied for $ a < -\frac{1}{4}$. Hence $a = -\frac{1}{4}, m_{(1^4)}=1,m_{(3,1)}=0,m_{(2,1^2)}=\frac{1}{4}, m_{(2^2)} = \frac{1}{16}$ and $\ell$ is a scalar of $\ev_{\left(\frac{1}{\sqrt{2}},\frac{1}{\sqrt{2}},0\right)}$.  
 \end{proof}
 
 \begin{lemma}\label{le:38Teil4}
 Let $\ell \in \left(\Sigma_{3,8}^{B_3}\right)^\ast$ be an extremal element. If $\alpha = 1, \gamma \geq 1$, then $\ell$ is a scalar of a point-evaluation in $(1,0,0), (1,1,1)$ or $\left(\sqrt{\frac{1}{2}+\sqrt{a+\frac{1}{4}}},\sqrt{\frac{1}{2}-\sqrt{a+\frac{1}{4}}},0\right)$, for $ -\frac{1}{4}\leq a \leq 0$.
 \end{lemma}
 \begin{proof}
 It is $S^{((1),(2))} \subset W_\ell$, i.e., for a pair $(b,c) \in \R^2 \setminus \{(0,0)\}$  $$be_1(\underline{X}^2)X_2X_3+cX_1^2X_2X_3 \in W_\ell$$ and their symmetrized products with elements in $H_{3,4}$ are contained in $W_\ell^{\langle 2 \rangle}$, i.e., \begin{align*}
    0=& b \frac{-a}{3}+cm_{(3,1)}, \\  0=&bm_{(3,1)}+\frac{c}{3}m_{(3,1)}.
 \end{align*} Inserting $\frac{ab}{3}=cm_{(3,1)}$ in the second equation gives $b \left( \frac{a}{9}+m_{(3,1)}\right) = 0.$ 
 \begin{itemize}
     \item[a)] We first assume that $b \neq 0$. Then $m_{(3,1)} = -\frac{a}{9}$. In this case we obtain from the positive semidefiniteness of the second matrix in (\ref{simpledmatrices}) that $$ 0 \leq \frac{2}{3}a^2-2m_{(3,1)} = \frac{2}{3}a(a + \frac{1}{3}).$$ Thus $a \in \{0,-\frac{1}{3}\}.$ If $a=0$ then $m_{(3,1)} =m_{(2,1^2)}=m_{(2^2)}=0$ and $\ell = \ev_{(1,0,0)}$. For $ a = -\frac{1}{3}$ it is $m_{(3,1)}=\frac{1}{27}, m_{(2,1^2)}= \frac{1}{3}, m_{(2^2)} = \frac{1}{9}$ and $\ell = \ev_{\left( \frac{1}{\sqrt{3}} ,\frac{1}{\sqrt{3}},\frac{1}{\sqrt{3}} \right)}.$
     \item[b)] In the remaining case $b = 0$ we can assume by linearity of $W_\ell$ that $c=1$, which implies $m_{(3,1)} = 0$. By the non-negativity of the last $1 \times 1$ matrix in (\ref{simpledmatrices}), i.e., $$0 \leq -\frac{a}{3}-\frac{4a^2}{3}+m_{(3,1)}$$ we obtain $-\frac{1}{4}\leq a \leq 0$. However, for any such $-\frac{1}{4} \leq a \leq 0$ it is $m_{(1^4)}=1,m_{(3,1)}=-a,m_{(2,1^2)}=a^2,m_{(2^2)}=0$ and $\ell = \ev_{ \left(\sqrt{\frac{1}{2}+\sqrt{a+\frac{1}{4}}},\sqrt{\frac{1}{2}-\sqrt{a+\frac{1}{4}}},0\right)}.$ 
 \end{itemize}
 \end{proof}
 
 \begin{lemma}\label{le:38Teil5}
  Let $\ell \in \left(\Sigma_{3,8}^{B_3}\right)^\ast$ be an extremal element. If $\alpha = \beta= 1$, then $\ell$ is a scalar of a point-evaluation in $\left( \sqrt{\frac{1+2\sqrt{1+3a}}{3}},\sqrt{\frac{1-\sqrt{1+3a}}{3}},\sqrt{\frac{1-\sqrt{1+3a}}{3}}\right)$, for $-\frac{1}{3} \leq a \leq 0$, or at \\
  $\left( \sqrt{\frac{1-2\sqrt{1+3b}}{3}}, \frac{\sqrt{1+\sqrt{1+3b}}}{3}, \frac{\sqrt{1+\sqrt{1+3b}}}{3}\right)$, for $- \frac{1}{3} \leq b \leq -\frac{1}{4}$. 
 \end{lemma}
 \begin{proof}
  If $\beta = 1$ then $S^{((2,1),\emptyset)} \subset W_\ell$, i.e., for a pair $(b,c) \in \R^2 \setminus \{(0,0)\}$  $$be_1(\underline{X}^2)(X_3^2-X_1^2)+c(X_2^2X_3^2-X_1^2X_2^2) \in W_\ell$$ and their symmetrized products with elements in $H_{3,4}$ are contained in $W_\ell^{\langle 2 \rangle}$, i.e.,  \begin{align*}
      0=& b \left( \frac{2}{3}+2a \right) + c \left( -\frac{1}{3}a-3m_{(3,1)} \right), \\ 0=& b \left( -\frac{1}{3}a-3m_{(3,1)}\right) + c \left( \frac{2}{3}a^2-2m_{(3,1)}\right).
  \end{align*}
  We distinguish two cases:
  \begin{itemize}
      \item[i)] If $b =0,c=1$ or if $b=1,c=0$ then $-\frac{1}{3}=a, m_{(3,1)} = \frac{1}{27}$ and $\ell = \ev_{\left( \frac{1}{\sqrt{3}},\frac{1}{\sqrt{3}},\frac{1}{\sqrt{3}}\right)}.$
      \item[ii)] We continue with the remaining case $b\neq 0$ and $c \neq 0$.
Since $W_\ell$ is a vector space we assume without loss of generality that $b=1$ and obtain $m_{(3,1)} = \frac{2}{9c}+\frac{2a}{3c}-\frac{a}{9}$ and $\frac{2(1+3a)(-3-2c+ac^2)}{9c}=0$. Hence $a=\frac{-1}{3}$ (then $\ell = \ev_{(\frac{1}{\sqrt{3}},\frac{1}{\sqrt{3}},\frac{1}{\sqrt{3}})}$) or $-3-2c+ac^2=0$. If $a=0$ then $c=-\frac{3}{2}$ and $m_{(3,1)}= -\frac{4}{27}$ which does not satisfy the positive semidefiniteness conditions. 
If $-\frac{1}{3} < a < 0$ then either $c = \frac{1}{a}-\sqrt{\frac{1+3a}{a^2}}$ or $c = \frac{1}{a}+\sqrt{\frac{1+3a}{a^2}}$.  \\
 In the first case it is $m_{(1^4)}=1,m_{(3,1)} = \frac{a\left(1+a\left(6+\sqrt{\frac{1+3a}{a^2}}\right)\right)}{9-9a\sqrt{\frac{1+3a}{a^2}}}, m_{(2,1^2)}=-a,m_{(2^2)}=a^2$. For any $-\frac{1}{3}< a < 0$ $\ell$ is the point-evaluation at  $\left( \sqrt{\frac{1+2\sqrt{1+3a}}{3}},\sqrt{\frac{1-\sqrt{1+3a}}{3}},\sqrt{\frac{1-\sqrt{1+3a}}{3}}\right)$. \\
 In the second case it is $m_{(1^4)}=1,m_{(3,1)} = \frac{a\left(1-a\left(-6+\sqrt{\frac{1+3a}{a^2}}\right)\right)}{9+9a\sqrt{\frac{1+3a}{a^2}}}, m_{(2,1^2)}=-a,m_{(2^2)}=a^2$. However, $m_{(3,1)} \geq 0$ is equivalent to $- \frac{1}{3} < a \leq -\frac{1}{4}.$ For any $- \frac{1}{3} < a \leq -\frac{1}{4}$ $\ell$ is the point-evaluation at $\left( \sqrt{\frac{1-2\sqrt{1+3a}}{3}}, \frac{\sqrt{1+\sqrt{1+3a}}}{3}, \frac{\sqrt{1+\sqrt{1+3a}}}{3}\right).$ 
  \end{itemize}
 \end{proof}

 \begin{proof}[Proof of Theorem \ref{Thm:DualEvenSymOct}]
 In Lemmas \ref{le:38Teil1}, \ref{le:38Teil2}, \ref{le:38Teil3}, \ref{le:38Teil4} and \ref{le:38Teil5} we have seen that the extremal rays in $\left( \Sigma_{3,8}^{B_3} \right)^\ast$ are all generated by point-evaluations. Those generators are the point-evaluations at elements in the set $$\left\{\left(a,\sqrt{1-a^2},0\right),{(b,c,c)} : \frac{1}{2} \leq a \leq 1, 0 \leq b \leq 1, c = \frac{1}{\sqrt{2}}\sqrt{(1-b^2)} \right\}.$$
 \end{proof}
 
 \begin{cor}
 The set of non-negative even symmetric ternary octics $\mathcal{P}_{3,8}^{B_3}$ is the convex cone generated by the following six forms \begin{align*}
    & e_1(\underline{X}^2)^4-3e_1(\underline{X}^2)^2e_2(\underline{X}^2),  -9e_1(\underline{X}^2)e_3(\underline{X}^2)+e_1(\underline{X}^2)^2e_2(\underline{X}^2), e_2(\underline{X}^2)^2-3e_1(\underline{X}^2)e_3(\underline{X}^2), \\ & e_1(\underline{X}^2)^2e_2(\underline{X}^2),   e_1(\underline{X}^2)e_3(\underline{X}^2), 3e_1(\underline{X}^2)e_3(\underline{X}^2)-4e_2(\underline{X}^2)^2+e_1(\underline{X}^2)^2e_2(\underline{X}^2) 
 \end{align*}
 and the following two families of forms
 \begin{align*}
      \left( \begin{array}{c}
          ae_1(\underline{X}^2)^4 +e_1(\underline{X}^2)e_2(\underline{X}^2), ae_1(\underline{X}^2)e_2(\underline{X}^2)     + e_2(\underline{X}^2)^2 : -\frac{1}{3} \leq a \leq 0
      \end{array}  \right)
 \end{align*}
 \end{cor}
 \begin{proof}
 These are precisely the sums of squares elements contained in the kernels of extremal rays of $\left(\Sigma_{3,8}^{B_3}\right)^\ast$. Since by Corollary \ref{COR Harris} $\Sigma_{3,8}^{B_3} = \mathcal{P}_{3,8}^{B_3}$, these are also precisely the elements in the boundary of the pointed convex cone $\mathcal{P}_{3,8}^{B_3}$. The claim follows from Minkowski's theorem.
 \end{proof}
 
 \begin{remark}
In \cite{harris1999real} Harris showed that $\Omega := \{(a,a,b),(0,a,b) : a,b \in \R_{\geq 0}\}$ is a test set for even symmetric ternary octics and used this as main ingredient in his proof of equality. In fact, our description in Theorem \ref{Thm:DualEvenSymOct} provides the subset of $\Omega$ consisting of all points of norm $1$, which was derived by describing $\left( \Sigma_{3,8}^{B_3}\right)^\ast$. \\
It is worth to point out that Harris result does not follow from Hilbert's equality case $\Sigma_{3,4}^{\mathfrak{S}_3} = \mathcal{P}_{3,4}^{\mathfrak{S}_3}$ for the symmetric group under canonical identification via the $\mathfrak{S}_3$-isomorphism $$
\abb{\Phi}{ H_{3,8}^{B_3}}{H_{3,4}^{\mathfrak{S}_3}}{\sum_{\alpha \in 2\N_0^3}c_\alpha \underline{X}^\alpha}{\sum_{\alpha \in 2\N_0^3} c_\alpha \underline{X}^{\frac{1}{2}\alpha}}.$$ For $g \in H_{3,4}^{\mathfrak{S}_3}$ it is $\Phi^{-1}(g)=g(X_1^2,X_2^2,X_3^2)$. Then $g$ is non-negative on the first orthant if and only if $\Phi^{-1}(g)$ is non-negative. However, the example $$f := e_1(\underline{X}^2)e_3(\underline{X}^2)=(X_1^2+X_2^2+X_3^2)(X_1^2X_2^2X_3^2) \in \mathcal{P}_{3,8}^{B_3}$$ with $\Phi (f)(-1,-1,1) = -1 < 0$ shows $\mathcal{P}_{3,4}^{\mathfrak{S}_3} \subsetneq \Phi (\mathcal{P}_{3,8}^{B_3})$.
  \end{remark}

 We demonstrate the stabilizing process from Theorem \ref{thm:ModulesStabilize} of $B_n$-Specht modules in $H_{n,d}$ for a fixed degree and large enough number of variables for even symmetric octics. This allows a uniform description of the sums of squares sets $\Sigma_{n,8}^{B_n}$, as stated in Corollary \ref{cor:ModulesStabilize}. \\
 We work with power means $p_i^{(n)}:= \frac{1}{n}\sum_{j=1}^nX_j^i \in \R[\underline{X}]^{\mathfrak{S}_n}$ instead of power sums. The upper index $n$ denotes that $p_i^{(n)}$ is a power mean in $n$ variables. Furthermore, for a partition $\lambda = (\lambda_1,\ldots,\lambda_l)$ we write $p_\lambda^{(n)}:= p_{\lambda_1}^{(n)}\cdot \ldots \cdot p_{\lambda_l}^{(n)}$. A reason for working with power means is that they are weighted, i.e., for any $i,n$ it is $p_{i}^{(n)}(1,1,\ldots,1) = 1, p_{i}^{(n)}(1,0,\ldots,0) = \frac{1}{n}$.
 \begin{lemma} \label{lem:Bnisotypicdecomp}
The $B_n$-isotypic decomposition of $H_{n,4}$ for $n \geq 4$ is \begin{align*} 2 \cdot S^{((n),\emptyset)} \oplus 2 \cdot S^{((n-1,1),\emptyset)} \oplus S^{((n-2,2),\emptyset)} \oplus 2 \cdot S^{((n-2),(2))}\oplus S^{((n-2),(1,1))} \oplus S^{((n-3,1),(2))} \oplus S^{((n-4),(4))}. \end{align*}
A symmetry adapted basis for $H_{n,4}$ realising the $B_n$-isotypic decomposition is generated by the following seven sets of polynomials
\begin{align*}
     S^{((n),\emptyset)} & : \left\{ p_{(4)}^{(n)},p_{(2^2)}^{(n)}\right\},& 
     S^{((n-1,1),\emptyset)} & :\left\{(X_n^2-X_1^2)p_{(2)}^{(n)}, X_n^4-X_1^4\right\}, \\
     S^{((n-2,2),\emptyset)} & :\left\{ (X_1^2-X_3^2)(X_2^2-X_4^2)\right\},& 
     S^{((n-2),(2))} & : \left\{X_{n-1}X_np_{(2)}^{(n)}, (X_{n-1}^2+X_n^2)X_{n-1}X_n\right\}, \\
     S^{((n-2),(1,1))} & : \left\{(X_n^2-X_{n-1}^2)X_{n-1}X_n\right\},&   S^{((n-4),(4))} & : \left\{X_1X_2X_3X_4\right\}, \\
     S^{((n-3,1),(2))} & : \left\{(X_n^2-X_1^2)X_{n-2}X_{n-1}\right\}.
\end{align*}
\end{lemma}
\begin{proof}
We determine the multiplicity of an irreducible $B_n$-module $S^{(\lambda,\mu)}$ in $H_{n,4}$ for a bipartition $(\lambda,\mu) \vdash n$ using Theorem \ref{thm:irreducible reps}. We can exclude some bipartitions immediately. The fundamental invariants of degree $\leq 4$ are of degree $2$ and $4$. Only $(\lambda,\mu)$ such that $\mu \vdash n_2$, with $n_2 \leq 4$ can occur, since a corresponding higher Specht polynomial has as a factor the monomial consisting of all products of the $X_i$'s, where $i$ ranges over the entries of the second bitableau. Furthermore, we only need to consider partitions $(\lambda,\mu)$ such that $|\mu|$ is even because a factor of the higher Specht polynomial is of degree $|\mu|$, while the additional factor has even degree. We can restrict us to bipartitions $(\lambda,\mu)$ such that there exist $(T,S) \in \SYT (\lambda,\mu)$ with $2\ch (T,S) +|\mu| \leq 4 $. Therefore a charge $\leq 2$ is necessary. We calculated all relevant higher Specht polynomials for $n \geq 4$:
\begin{align*}
     S^{((n),\emptyset)} & : \left\{1\right\},& 
     S^{((n-1,1),\emptyset)} & :\left\{ X_n^2-X_1^2,\frac{1}{n}\sum_{i=2}^{n-1}X_i^2(X_n^2-X_1^2)\right\}, \\
     S^{((n-2,2),\emptyset)} & : \left\{(X_1^2-X_3^2)(X_2^2-X_4^2)\right\}, &
     S^{((n-2),(2))} & : \left\{X_{n-1}X_n, \frac{1}{n-2}(X_1^2+\ldots+X_{n-2}^2)X_{n-1}X_n\right\}, \\
     S^{((n-2),(1,1))} & : \left\{(X_n^2-X_{n-1}^2)X_{n-1}X_n\right\},& 
     S^{((n-4),(4))} & : \left\{X_1X_2X_3X_4\right\}, \\
     S^{((n-3,1),(2))} & : \left\{(X_n^2-X_1^2)X_{n-2}X_{n-1}\right\}.
\end{align*}
Multiplying them with the weighted power sums gives the $B_n$-symmetry adapted basis of $H_{n,4}$.
However, since \begin{align*}
X_n^4-X_1^4 &\in \langle p_2^{(n)}(X_n^2-X_1^2),\frac{1}{n}\sum_{i=2}^{n-1}X_i^2(X_n^2-X_1^2) \rangle_\R,  \\
(X_{n-1}^2+X_n^2)X_{n-1}X_n &\in \langle
     p_2^{(n)} X_{n-1}X_n, \frac{1}{n-2}(X_1^2+\ldots+X_{n-2}^2)X_{n-1}X_n \rangle_\R,
\end{align*} we can work with the above mentioned symmetry adapted basis. 
\end{proof}
\begin{thm} \label{thm:EvenSymSOSOctics}
Let $n \geq 4$. An even symmetric $n$-ary octic $f \in H_{n,8}^{B_n}$ is a sum of squares if and only if there exist positive semidefinite matrices $A^{((n),\emptyset)},A^{((n-1,1),\emptyset)},A^{((n-2,2),\emptyset)},A^{((n-2),(2))} \in \R^{2 \times 2}$ and $A^{((n-2),(1,1))},A^{((n-4),(4))},A^{((n-3,1),(2))} \in \R_{\geq 0}^{1 \times 1}$ such that \begin{align*}
    \mathfrak{f} & = \langle A^{((n),\emptyset)}B^{((n),\emptyset)} \rangle +\langle A^{((n-1,1),\emptyset)}B^{((n-1,1),\emptyset)} \rangle+\langle A^{((n-2,2),\emptyset)}B^{((n-2,2),\emptyset)}\rangle \\ & +\langle A^{((n-2),(2))}B^{((n-2),(2))} \rangle +A^{((n-2),(1,1))}B^{((n-2),(1,1))}+A^{((n-4),(4))}B^{((n-4),(4))} \\ & +A^{((n-3,1),(2))}B^{((n-3,1),(2))},  
\end{align*} 
where 
\begin{align*}
    B^{((n),\emptyset)} & :=  \left( \begin{array}{cc} p_{(4^2)}^{(n)} & p_{(4,2^2)}^{(n)} \\ p_{(4,2^2)}^{(n)}&  p_{(2^4)}^{(n)} \end{array}\right),  \\
    B^{((n-1,1),\emptyset)} & :=  \left( \begin{array}{cc} p_{(4,2^2)}^{(n)} -p_{(2^4)}^{(n)} & p_{(6,2)}^{(n)}-p_{(4,2^2)}^{(n)} \\ p_{(6,2)}^{(n)}-p_{(4,2^2)}^{(n)} & p_{(8)}^{(n)}-p_{(4^2)}^{(n)} \end{array} \right), \\
     B^{((n-2,2),\emptyset)}   & :=  \left(  \frac{-n+1}{n^2}p_{(8)}^{(n)}+\frac{4n-4}{n^2}p_{(6,2)}^{(n)}+\frac{n^2-3n+3}{n^2}p_{(4^2)}^{(n)}-2p_{(4,2^2)}^{(n)}+p_{(2^4)}^{(n)}\right), \\
     B^{((n-2),(2))} & := \left( \begin{array}{cc}
        p_{(2^4)}^{(n)} - \frac{1}{n}p_{(4,2^2)}^{(n)}  &  2p_{(4,2^2)}^{(n)}-\frac{2}{n}p_{(6,2)}^{(n)} \\
        2p_{(4,2^2)}^{(n)}-\frac{2}{n}p_{(6,2)}^{(n)}  & 2p_{(6,2)}^{(n)} +2p_{(4^2)}^{(n)}-\frac{4}{n}p_8^{(n)} 
     \end{array} \right), \\
      B^{((n-2),(1,1))} & := \left( \begin{array}{c}  p_{(6,2)}^{(n)}-p_{(4^2)}^{(n)} \end{array}\right), \\
       B^{((n-4),(4))} &:= \left( \begin{array}{c} p_{(2^4)}^{(n)}-\frac{6}{n}p_{(4,2^2)}^{(n)}+\frac{3}{n^2}p_{(4^2)}^{(n)}+\frac{8}{n^2}p_{(6,2)}^{(n)}-\frac{6}{n^3}p_{(8)}^{(n)}) \end{array}\right), \\
        B^{((n-3,1),(2))} &:= \left( \begin{array}{c} \frac{2}{n^2}p_{(8)}^{(n)}-\frac{2n+2}{n^2}p_{(6,2)}^{(n)}-\frac{1}{n}p_{(4^2)}^{(n)}+\frac{n+3}{n}p_{(4,2^2)}^{(n)}-p_{(2^4)}^{(n)} \end{array} \right). 
\end{align*}
\end{thm}
\begin{proof}
The matrices $B^{(i)}$ are the matrices containing the symmetrized products of the symmetry adapted basis of the $B_n$-module $H_{n,4}$ from Lemma \ref{lem:Bnisotypicdecomp}. By Theorem \ref{THM Decomp} any invariant sums of squares form has such a representation.
\end{proof}

We observe that for $n \geq 4$ the $\R$-vector spaces $$ H_{n,8}^{B_n} = \langle p_{(2^4)}^{(n)},p_{(4,2^2)}^{(n)},p_{(4^2)}^{(n)},p_{(4,2)}^{(n)},p_{(6,2)}^{(n)},p_{8}^{(n)}\rangle_\R$$ are of the same dimension and thus can be identified. We identify the vector spaces with respect to the isomorphisms $$ p_{\lambda}^{(n)} \mapsto p_{\lambda}^{(m)}$$ for $n,m \in \N_{\geq 4}$. In \cite{blekrie} Blekherman and the second author studied symmetric quartic forms and defined a limit set consisting of all $\left( p_{\lambda}^{(n)}\right)_{n\geq 4}$. They showed that for symmetric quartics the limit sets of the cones of symmetric sums of squares and non-negative quartics are equal. As a first step towards a similar result in the $B_n$ case we provide a classification of the limit set of the cones of even symmetric octics that are sums of squares. 
\begin{remark} 
 The matrices in Theorem \ref{thm:EvenSymSOSOctics} have the following limits for $n \rightarrow \infty$
\begin{align*}
    \mathcal{B}^{((n),\emptyset)} & :=  \left( \begin{array}{cc} \mathfrak{p}_{(4^2)} & \mathfrak{p}_{(4,2^2)} \\ \mathfrak{p}_{(4,2^2)}&  \mathfrak{p}_{(2^4)}\end{array}\right)  \\
    \mathcal{B}^{((n-1,1),\emptyset)} & :=  \left( \begin{array}{cc} \mathfrak{p}_{(4,2^2)} -\mathfrak{p}_{(2^4)} & \mathfrak{p}_{(6,2)}-\mathfrak{p}_{(4,2^2)} \\ \mathfrak{p}_{(6,2)}-\mathfrak{p}_{(4,2^2)} & \mathfrak{p}_{(8)}-\mathfrak{p}_{(4^2)} \end{array} \right), \\
     \mathcal{B}^{((n-2,2),\emptyset)} & := \left( \begin{array}{c} \mathfrak{p}_{(4^2)} -2\mathfrak{p}_{(4,2^2)}+\mathfrak{p}_{(2^4)} \end{array}\right), \\
     \mathcal{B}^{((n-2),(2))} & := \left( \begin{array}{cc}
        \mathfrak{p}_{(2^4)}   & 2 \mathfrak{p}_{(4,2^2)} \\
       2\mathfrak{p}_{(4,2^2)}   &  2 \mathfrak{p}_{(6,2)}+2\mathfrak{p}_{(4^2)}
     \end{array} \right), \\
      \mathcal{B}^{((n-2),(1,1))} & :=  \left( \begin{array}{c} \mathfrak{p}_{(6,2)}-\mathfrak{p}_{(4^2)} \end{array}\right), \\
       \mathcal{B}^{((n-4),(4))} & := \left( \begin{array}{c} \mathfrak{p}_{(2^4)} \end{array}\right), \\
        \mathcal{B}^{((n-3,1),(2))} & :=  \left( \begin{array}{c} \mathfrak{p}_{(4,2^2)}-\mathfrak{p}_{(2^4)} \end{array} \right). 
\end{align*}
\end{remark}
\begin{cor} \label{COR Limit Octics}
An even symmetric homogeneous octic limit sum of squares inequality $\mathfrak{f}$ has the form
\begin{align*}
    \mathfrak{f} & = \alpha_1 \mathfrak{p}_{(4^2)} + 2 \alpha_2 \mathfrak{p}_{(4,2^2)}+ \alpha_3 \mathfrak{p}_{(2^4)}\\
    & + \beta_1 (\mathfrak{p}_{(4,2^2)}-\mathfrak{p}_{(2^4)})+2\beta_2 (\mathfrak{p}_{(6,2)}-\mathfrak{p}_{(4,2^2)})+\beta_3 (\mathfrak{p}_{(8)}-\mathfrak{p}_{(4^2)}) \\
    & + \delta (\mathfrak{p}_{(6,2)}-\mathfrak{p}_{(4^2)}), 
\end{align*}
where $\left( \begin{array}{cc} \alpha_1 & \alpha_2 \\ \alpha_2 & \alpha_3 \end{array}\right), \left( \begin{array}{cc} \beta_1 & \beta_2 \\ \beta_2 & \beta_3 \end{array}\right), \left( \delta \right)$ are positive semidefinite real matrices. 
\end{cor}
\begin{proof}
We observe that an invariant limit sum of squares coming from the irreducible representation $S^{((n-2,2),\emptyset)}$, i.e., $\mathfrak{p}_{(4^2)}-2\mathfrak{p}_{(4,2^2)}+\mathfrak{p}_{(2^4)}$, is contained in the first line.
The limit sum of squares $\mathfrak{p}_{(2^4)}$ from $S^{((n-4),(4))}$ is contained in the first line, while the limit form from $S^{((n-3,1),(2))}$, i.e., $\mathfrak{p}_{(4,2^2)}-\mathfrak{p}_{(2^4)}$ is contained in the second line for $\beta_1 = 1.$ Furthermore, $$ (\alpha,\beta) \left( \begin{array}{cc}
        \mathfrak{p}_{(2^4)}   & 2 \mathfrak{p}_{(4,2^2)} \\
       2\mathfrak{p}_{(4,2^2)}   &  2 \mathfrak{p}_{(6,2)}+2\mathfrak{p}_{(4^2)}
     \end{array} \right) (\alpha,\beta)^T = 2\beta^2 (\mathfrak{p}_{(6,2)}-\mathfrak{p}_{(4^2)} ) + \langle \left( \begin{array}{cc}
    4\beta^2      &2\alpha\beta  \\
    2\alpha\beta      & \alpha^2
     \end{array} \right), \left( \begin{array}{cc} \mathfrak{p}_{(4^2)} & \mathfrak{p}_{(4,2^2)} \\ \mathfrak{p}_{(4,2^2)}&  \mathfrak{p}_{(2^4)}\end{array}\right) \rangle.$$
\end{proof}
\begin{remark}
Let $\Sigma_{\infty,8}^{B_\infty}$ denote the cone consisting of all limit forms from Corollary \ref{COR Limit Octics}, $\Sigma_{\infty,4}^{\mathfrak{S}_\infty}$ the limit cone of symmetric sums of squares of degree $4$ in \cite{blekrie} and $\Phi : H_{\infty,8}^{B_\infty}\rightarrow H_{\infty,4}^{\mathfrak{S}_\infty}$ be the canonical $\mathfrak{S}_{\infty}$-homomorphism. The cones $\Phi \left( \Sigma_{\infty,8}^{B_\infty}\right)$ and $\Sigma_{\infty,4}^{\mathfrak{S}_\infty} = \mathcal{P}_{\infty,4}^{\mathfrak{S}_\infty}$ are different. This is not surprising, since the cone $\Phi \left( \mathcal{P}_{\infty,8}^{B_\infty}\right)$ can be identified with the limit of all symmetric forms that are non-negative on the positive orthant (compare with Polya's Nichtnegativenstellensatz \cite{polya1928positive}). 
\end{remark}
 It is a question for further studies to determine the relation between the limit cones of even symmetric sums of squares and non-negatives octics.

\subsection{Forms invariant under $D_n$}\label{secdn}
It is a natural question to wonder, to what extend Harris' result on ternary forms invariant under $B_3$ carries over to the slightly smaller group $D_3$. As is shown in the following theorem we obtain equality between the sets $\Sigma_{3,8}^{D_3}$ and $\mathcal{P}_{3,8}^{D_3}$. Furthermore, we prove that $\mathcal{P}_{4,4}^{D_4}$ is a simplicial cone which gives a test set for non-negativity consisting of three points. We prove that for quaternary quartics invariant under $D_4$ we also have that non-negativity implies a sums of squares representation. We conclude with a full characterization of the non-negativity versus sums of squares question for forms invariant under $D_n$.
\begin{thm}\label{thm:harrisd3}
The sets of non-negative and sums of squares ternary octics invariant under $D_3$ are equal, i.e.,
$\Sigma_{3,8}^{D_3}=\mathcal{P}_{3,8}^{D_3}$.
\end{thm}
\begin{proof}
The invariant ring $\R[X_1,X_2,X_3]^{D_3}=\R[p_2,e_3,p_4]$ is a polynomial ring in the symmetric polynomials $p_2, e_3$ and $p_4$. A vector space basis of $H_{3,8}^{D_3}$ is given by $\left( p_{(2^4)},p_{(4,2^2)},p_{(4^2)},p_{2}e_3^2 \right)$. In Remark \ref{rmk:isotypic decomp H38Bn} we have seen that $H_{3,8}^{B_3} = \langle p_{(2^4)}, p_{(4,2^2)}, p_{(4^2)}, p_{(6,2)} \rangle_\R.$ The functions $p_6$ and $e_3^2$ occur linearly in the following identity for symmetric functions in three variables $$ p_{(2^3)}-3p_{(4,2)}+2p_6-6e_3^2=0. $$ Hence we deduce that $H_{3,8}^{D_3} = H_{3,8}^{B_3}$. The claim follows by Corollary \ref{COR Harris}. 
\end{proof}

\begin{remark}
We have the same conical generators and test set as in the $B_3$ case for non-negative ternary octics invariant under $D_3$, i.e., a form $f\in H_{3,8}^{D_3}$ is non-negative if and only if $f(y)\geq 0$ for all $y \in \{ (a,a,b),(0,a,b) : a,b \in \R_{\geq 0} \}$.
\end{remark}

In the following we study quaternary quartics invariant under $D_4$.
\begin{lemma} \label{le:D4isodecomp}
The $D_4$-module $H_{4,2}$ has the isotypic decomposition \begin{align*}
H_{4,2} & =  S^{((4),\emptyset)} \oplus S^{((3,1),\emptyset)} \oplus S_1^{((2),(2))} \oplus S_2^{((2),(2))}.    
\end{align*}
The symmetry adapted basis which realizes the $D_4$-module decomposition of $H_{4,2}$ is the following:
\begin{align*}
 S^{((4),\emptyset)} & : \left\{p_{(2)}\right\},& 
 S^{((3,1),\emptyset)} & :\left\{ X_4^2-X_1^2\right\}, \\
 S^{((2),(2))}_1   & : \left\{X_1X_2+X_3X_4\right\}, &
S^{((2),(2))}_2   & : \left\{X_1X_2-X_3X_4\right\}.
\end{align*}
\end{lemma}

\begin{proof}
By Theorem \ref{thm:irreducible reps} we have to determine the multiplicity of the irreducible $D_4$-modules labelled by bipartitions $(\lambda,\mu) \vdash 4$ of the form $|\lambda|\geq |\mu|$. We are just interested in higher Specht polynomials of degree $0$ or $2$, since the $D_4$ fundamental invariant of degree $\leq 2$ is $p_2$. Thus it must be $|\mu|\in \{0,2\}$. If $\mu = \emptyset$, then the partitions $((4),\emptyset),((3,1),\emptyset)$ have one standard bitableau whose charge is at most $1$, i.e., they occur precisely once in $H_{4,2}$. Any occurring module labelled by $(\lambda,\mu)$ with $|\mu|=2$ must have a standard bitableau with index $(0,0,0,0)$. This can only occur if the word equals $(1,2,3,4)$. Thus just the bipartition $((2),(2))$ has a standard bitableau with zero charge. By Theorem \ref{thm:irreducible reps} the module $S^{((2),(2))}$ decomposes into two irreducible $D_4$-modules $S_1^{((2),(2))}$ and $S_2^{((2),(2))}$. We calculated the relevant higher Specht polynomials according to Theorem \ref{thm:irreducible reps}
$$\left\{1,\, X_4^2-X_1^2,\, X_1X_2+X_3X_4,\, X_1X_2-X_3X_4\right\},$$ and find accordingly the polynomials above.
\end{proof}
\begin{cor}
A $D_4$-invariant quaternary quartic $f \in H_{4,4}^{D_4}$ is a sum of squares if and only if there exist positive numbers $A^{(1)},A^{(2)},A^{(3)},A^{(4)} \in \R_{\geq 0}$ such that $$ f = A^{(1)}B^{(1)}+ A^{(2)}B^{(2)}+A^{(3)}B^{(3)}+A^{(4)}B^{(4)},$$ where  
\begin{align*}
    B^{((4),\emptyset)} &:= \left( p_{(2^2)} \right),&
    B^{((3,1),\emptyset)} &:= \left( \frac{2}{3}p_{(4)}-\frac{1}{6}p_{(2^2)}\right), \\
    B_1^{((2),(2))} &:= \left( \frac{1}{6}p_{(2^2)}-\frac{1}{6}p_{(4)}+2e_4 \right), &
    B_2^{((2),(2))} &:= \left( \frac{1}{6}p_{(2^2)}-\frac{1}{6}p_{(4)}-2e_4 \right).
\end{align*}
\end{cor}
\begin{proof}
The matrices $B^{(i)}$ are obtained by calculating the Reynolds operator evaluated at squares of the symmetry adapted basis of the irreducible $D_4$-modules from Lemma \ref{le:D4isodecomp}. By Theorem \ref{THM Decomp} any invariant sum of squares form has such a representation.
\end{proof}

\begin{thm}
\label{thm:D4simplicialCone} The dual cone of $D_4$-invariant sum of squares quartics is a simplicial cone with the following description
$\left(\Sigma_{4,4}^{D_4}\right)^\ast = \cone \left\{ \ev_{(1,0,0,0)},\ev_{(1,1,1,-1)}, \ev_{(1,1,1,1)} \right\}.$
\end{thm} 
\begin{proof}
Let $\ell \in \left( \Sigma_{4,4}^{D_4}\right)^\ast$ denote an extremal element. Let \begin{align*}
    W_\ell := \alpha \cdot S^{((4),\emptyset)} \oplus \beta \cdot S^{((3,1),\emptyset)} \oplus \gamma \cdot S_1^{((2),(2))} \oplus \delta \cdot S_2^{((2),(2))} 
\end{align*} denote the $D_4$-submodule of $H_{4,2}$ which is the kernel of the associated quadratic form, for $\alpha,\beta,\gamma,\delta \in \{0,1\}$. Now, we show that $\ell$ must be a scalar of one of the three point-evaluations above, respectively that $W_\ell^{\langle 2\rangle}$ must have one of the points as a zero. \\
Since $p_{(2^2)}$ is not contained in the boundary of $\Sigma_{4,4}^{D_4}$ it must be $\alpha = 0$. Furthermore, $\dim_\R W_\ell^{\langle 2 \rangle} = 2$ and therefore we have that precisely two of the parameters are non-zero, because the symmetrized squares of the symmetry adapted basis elements belonging to the $D_4$-modules $S^{((3,1),\emptyset)}$, $S_1^{((2),(2))}$ and $S_2^{((2),(2))}$ are linearly independent. 
\begin{itemize}
    \item[i)] We start by examining the case $\gamma = \delta  = 1$. Then $\ell (e_4) = 0, \ell (p_{(2^2)}) = \ell (p_{(4)})$ and $$W_\ell^{\langle 2 \rangle } = \langle e_4,p_{(2^2)}-p_{(4)} \rangle_\R.$$ $W_\ell^{\langle 2 \rangle}$ has the root $(1,0,0,0)$.
\end{itemize}
We proceed with the cases $\gamma = \beta = 1$ or $\beta = \delta = 1$. 
\begin{itemize}
    \item[ii)] We notice that if $\gamma = \beta = 1$ then $$W_\ell = \langle X_4^2-X_1^2, X_1X_2+X_3X_4 \rangle_{D_4},$$ but all elements in $W_\ell$ have the common root $(1,1,1,-1)$.
    \item[iii)] If $\beta = \delta = 1$ then $$W_\ell = \langle X_4^2-X_1^2,X_1X_2-X_3X_4\rangle_{D_4}$$ with the common root $(1,1,1,1)$.
\end{itemize}
\end{proof}

\begin{cor}\label{d4}
The set of non-negative and sums of squares quaternary quartics invariant under $D_4$ are equal, i.e., it is
$\Sigma_{4,4}^{D_4} = \mathcal{P}_{4,4}^{D_4}$.
\end{cor}
This does not already follow from $\Sigma_{4,4}^{B_4} = \mathcal{P}_{4,4}^{B_4}$ in \cite{goel2017analogue}, because  $H_{4,4}^{D_4}\setminus H_{4,4}^{B_4}\neq \emptyset$.
\begin{proof}
By Theorem \ref{thm:D4simplicialCone} the cone $\left( \Sigma_{4,4}^{D_4} \right)^\ast$ is generated by point-evaluations. Hence any extremal ray in $\left( \Sigma_{4,4}^{D_4} \right)^\ast$ is spanned by a point-evaluation. The claim follows from Corollary \ref{cor:Sigma=P}.
\end{proof}

By reformulating Theorem \ref{thm:D4simplicialCone} we obtain the following very simple test set for $D_4$-quartics:
\begin{cor}
A form $f = a(X_1^2+X_2^2+X_3^2+X_4^2)^2+b(X_1^4+X_2^4+X_3^4+X_4^4)+cX_1X_2X_3X_4$, with $a,b,c \in \R$, is non-negative if and only if $f(z) \geq 0$ for all $z \in \{ (1,0,0,0),(1,1,1,-1),(1,1,1,1)\}.$
\end{cor}
\begin{proof}
An invariant form $f \in H_{4,4}^{D_4}$ is non-negative if and only if $\ell (f) \geq 0$ for any $\ell$ in $\left( \mathcal{P}_{4,4}^{D_4}\right)^\ast.$ By Corollary \ref{d4} it is $\left( \mathcal{P}_{4,4}^{D_4}\right)^\ast = \left(\Sigma_{4,4}^{D_4}\right)^\ast$. The claim follows from Theorem \ref{thm:D4simplicialCone}.
\end{proof} 

\begin{cor}
The convex cone $\mathcal{P}_{4,4}^{D_4}$ of non-negative $D_4$-quartics is a simplicial cone generated by $$4p_{(4)}-p_{(2^2)}, p_{(2^2)}-p_{(4)}+12e_4, p_{(2^2)}-p_{(4)}-12e_4.$$ 
\end{cor}
\begin{proof}
The sets $\mathcal{P}_{4,4}^{D_4}$ and $\Sigma_{4,4}^{D_4}$ are equal by Corollary \ref{d4}. The boundary of $\Sigma_{4,4}^{D_4}$ is equal to the union of all kernels of extremal elements in $\left(\Sigma_{4,4}^{D_4}\right)^\ast$ intersected with $\Sigma_{4,4}^{D_4}$. The above generators are precisely the invariant sums of squares contained in the kernels of the three extremal rays in Theorem \ref{thm:D4simplicialCone}.
\end{proof}

The results from the previous two subsections allow to conclude the following classification for the equivariant non-negativity versus sums of squares question for the reflection group $D_n$. 
\begin{thm}\label{thm:chardn}
$\Sigma_{n,2d}^{D_n} = \mathcal{P}_{n,2d}^{D_n}$ if and only if $(n,2d) \in \{ (n,2), (n,4) , (3,8)\}.$
\end{thm}
\begin{proof}
Suppose that there exists $f\in \mathcal{P}_{n,2d}^{B_n}\setminus \Sigma_{n,2d}^{B_n}$. This implies $f\in \mathcal{P}_{n,2d}^{D_n}\setminus \Sigma_{n,2d}^{D_n}$.
Therefore, we can directly rely on the classification carried out in \cite{goel2017analogue} and only need to consider those cases specifically, where all even symmetric positive semidefinite forms are sums of squares. These  are only the following non-trivial cases: $(n,2d) \in \{ (3,8), (n,4) \}$. But we have shown in Theorem \ref{thm:harrisd3} that in the case  $(3,8)$ the equality does survive, and while following Corollary \ref{d4} it does also for $(4,4)$. Furthermore, if $n > 4$ then the invariant quartics with respect to $B_n$ are precisely the invariant quartics with respect to $D_n$ as $H_{n,4}^{B_n} = \langle p_{(2^2)}, p_{(4)}\rangle_\R = H_{n,4}^{D_n}$ for $n \geq 5$, which finishes the proof.
\end{proof}
\subsection{LMIs and non-negativity testing}

In general testing non-negativity of a polynomial in more than two variables is already for quartics an NP-hard problem (see e.g. \cite{blum1998complexity} or \cite{murty1985some}). On the other hand, certifying that a given polynomial is a sum of squares can be done with so called semidefinite programming. Although the complexity status of this procedure in the Turing or in the real numbers model is not yet known (see \cite{complexsdp}) SDPs can be solved numerically in polynomial time to a given accuracy via the ellipsoid algorithm and  
this approach generally provides a tractable way to certify that a polynomial is non-negative, if it is a sum of squares. The feasible region of a semidefinite program is given by a linear matrix inequality (LMI), i.e., an inequality of the form $A_0+x_1A_1+x_2A_2+\ldots+x_nA_n\succeq 0$,
where $A_0,\ldots, A_{n}$ are real symmetric matrices all of the same size and $x_1,\ldots,x_n$ are supposed to be real scalars. The set of all $x\in\R^n$ satisfying a given LMI is called a \emph{spectrahedron}. For every $f\in H_{n,2d}$ one can construct an LMI (\cite{powers1998algorithm}) which possesses a solution if and only if $f$ is a sum of squares. The corresponding spectrahedron is called the \emph{Gram spectrahedron}  of $f$\cite{chua2016gram}, and it represents in fact all possible ways to decompose $f$ into sums of squares. Accordingly, it is non-empty if and only if $f$ is a sum of squares.  The results presented in the article can be directly transferred into the setup of symmetry adapted Gram-spectrahedra, which were, for example, recently studied by \cite{shankar}. 

\begin{thm}
Let $G$ be a finite reflection group and consider $f\in H_{n,2d}^{G}$ and $\theta_1,\ldots,\theta_l$ be all non $G$-isomorphic irreducible representations. Then the Gram spectrahedron of $f$ can be defined by a block diagonal matrix, consisting of $l$ blocks $B_1,\ldots,B_l$ and the size of the block $B_i$ equals
$$\sum_{k=0}^dN(d-k)\cdot h_k^{\vartheta_i}.$$
In particular, in the case $G\in\{A_{n-1},B_n,D_n\}$ the size of the matrix is independent of $n$, for large $n$
\end{thm}
\begin{proof}
This follows from choosing a symmetry adapted basis of $H_{n,d}$ and Corollary \ref{cor:bound}. When $G\in\{A_{n-1},B_n,D_n\}$ the stabilization follows from Corollary  \ref{cor:ModulesStabilize}
\end{proof}

A convex set which is not a spectrahedron but can be obtained as the projection of higher dimensional spectrahedron is called \emph{spectrahedral shadow}. Following a question by Nemirovski, which convex sets can be represented as projections of spectrahedra, Scheiderer \cite{scheiderer} showed that the cones of non-negative forms in general are not spectrahedral shadows. In the next theorem we give some examples of invariant non-negative forms, which form spectrahedral shadows.

\begin{thm} 
For all $n$ the families of cones $\mathcal{P}_{n,4}^{\mathfrak{S}_n}$, $\mathcal{P}_{n,6}^{B_n}$, $\mathcal{P}_{n,8}^{B_n}$ and $\mathcal{P}_{n,10}^{B_n}$ are  spectrehedral shadows. Moreover, for forms in any of these families, there exists an LMI of size $O(n^3)$ certifying the non-negativity. 
\end{thm}

\begin{proof}
For $n \leq 2$ this is trivial, and in the case $n=3$ this follows either from Hilbert's Theorem in the $\mathfrak{S}_3$ case or from Harris' result \ref{COR Harris} in the $B_3$ case.
So we assume $n\geq 4$. By the half-degree principle, an element $f \in H_{n,4}^{\mathfrak{S}_n}$ is non-negative on $\R^n$ if and only if for any partition $\lambda \vdash n$ of length $2$ the form $f^\lambda \in H_{2,4}$ is non-negative on $\R^2$, where $f^\lambda(x,y):=f(x,\ldots,x,y,\ldots,y)$ and $x$ occurs precisely $\lambda_1$ times and $y$ $\lambda_2$ times. Notice that each $f^{\lambda}$ is non-negative if and only if it is a sum of squares, i.e., if we have $f^{\lambda}\in\Sigma_{2,4}$. If we denote by $\Phi^\lambda$ the linear map $f \mapsto \tilde{f}^\lambda (x,y)$ and if $\lambda^1,\ldots,\lambda^m$ are all partitions of $n$ with length $2$ then $$\mathcal{P}_{n,4}^{\mathfrak{S}_n} = \bigcap_{i=1}^m \left(\Phi^{\lambda^i} \right) ^{-1}(\Sigma_{2,4} )$$
which proves the claim in the $\mathfrak{S}_n$ case. 
Using the half-degree principle \cite[Theorem 3.1]{riener2016symmetric} for $B_n$ and considering instead of $f(\underline{X}) \in \R[\underline{X}]^{B_n}$ the form $f(\sqrt{|X_1|},\ldots,\sqrt{|X_n|})\in \R[\underline{X}]^{\mathfrak{S}_n}$, one can argue analogously with slight modifications. 
\end{proof}
\begin{remark}
In the case of symmetric polynomials, the above statement was implicitly already stated in \cite[Theorem 5.5]{riener2013exploiting} for symmetric quartic forms, albeit without mentioning of the term spectrahedral shadow.
\end{remark}
The core of the proof above is the reduction to bivariate forms via test sets. 

\begin{thm}
For the families of cones $\mathcal{P}_{n,6}^{\mathfrak{S}_n}$, $\mathcal{P}_{n,12}^{B_n}$ and $\mathcal{P}_{n,14}^{B_n}$ membership can be decided with $O(n^3)$ many LMIs, each of which has size bounded independent of $n$. 
\end{thm}
\begin{proof}
 Using the half-degree principle \cite[Theorem 3.1]{riener2016symmetric} one finds that  membership in each of the above-mentioned cones can be decided by reducing the $O(n^3)$ many ternary forms, similarly to the proof above. For each of these  ternary forms, one can decide non-negativity individually. De Klerk and Pasechnik \cite{deklerk} provided a construction to decide non-negativity of a ternary form of degree $2d$ by means of  $d/4$ LMIs each of which is polynomial in $d$. Combining their construction with the arguments above thus yields an LMI of the announced size.
\end{proof}

\subsubsection*{Acknowledgements} The authors would like to thank Jose Acevedo for a simplification in Corollary \ref{COR Limit Octics} and Greg Blekherman and Markus Schweighofer for helpful insights. Furthermore, the comments of two referees were very helpful to improve the paper.

\end{document}